\documentclass[10pt, a4paper]{amsart}
\usepackage{amsmath, amssymb,amsthm}
\usepackage[extra]{tipa}
\usepackage{lettrine}
\usepackage[Sonny]{fncychap}
\usepackage[english]{babel}
\usepackage{amsmath, amssymb,amsthm}
\usepackage[all]{xy}
\usepackage{tabularx}
\usepackage[applemac]{inputenc}
\usepackage{graphicx}
\usepackage[dvipsnames]{xcolor}
\usepackage{hyperref}
\usepackage[applemac]{inputenc}
\usepackage{enumitem}
\usepackage{geometry}
\usepackage{prerex}
\usetikzlibrary{fit}
\usepackage{pdflscape}
\usepackage{tikz-cd} 
\usepackage{mathtools}

\makeatletter
\def\oversortoftilde#1{\mathop{\vbox{\m@th\ialign{##\crcr\noalign{\kern3\p@}%
      \sortoftildefill\crcr\noalign{\kern3\p@\nointerlineskip}%
      $\hfil\displaystyle{#1}\hfil$\crcr}}}\limits}

\def\sortoftildefill{$\m@th \setbox\z@\hbox{$\braceld$}%
  \braceld\leaders\vrule \@height\ht\z@ \@depth\z@\hfill\braceru$}

\makeatother

\usepackage{mathrsfs}

\DeclareMathOperator{\genus}{genus}
\DeclareMathOperator{\unr}{unr}
\DeclareMathOperator{\et}{et}
\DeclareMathOperator{\Ig}{Ig}
\DeclareMathOperator{\sub}{sub}
\DeclareMathOperator{\quot}{quot}
\DeclareMathOperator{\End}{End}

\DeclareMathOperator{\Ker}{Ker}

\DeclareMathOperator{\ord}{ord}
\DeclareMathOperator{\Gal}{Gal}

\DeclareMathOperator{\Div}{Div}

\DeclareMathOperator{\ab}{ab}
\DeclareMathOperator{\Hom}{Hom}
\DeclareMathOperator{\Symb}{Symb}

\DeclareMathOperator{\Aut}{Aut}
\DeclareMathOperator{\Card}{Card}

\DeclareMathOperator{\cusps}{cusps}
\DeclareMathOperator{\Frob}{Frob}
\DeclareMathOperator{\SL}{SL}
\DeclareMathOperator{\GL}{GL}

\newcommand{\cf}{\textit{cf. }}
\newcommand{\ie}{\textit{i.e. }}

\newcommand{\eg}{\textit{e.g. }}
\theoremstyle{definition}

\newtheorem{rems}{Remarks}[section]
\theoremstyle{plain}
\newtheorem{thm}{Theorem}[section]
\newtheorem{defn}[thm]{Definition}
\newtheorem{question}[thm]{Question}
\newtheorem{conj}[thm]{Conjecture}
\newtheorem{lem}[thm]{Lemma}

\newtheorem{prop}[thm]{Proposition}

\title{On the Mazur--Tate conjecture for prime conductor and Mazur's Eisenstein ideal}

\begin{document}
\maketitle
\author{Emmanuel Lecouturier \footnote{Email: elecoutu@mail.tsinghua.edu.cn \\ Yau Mathematical Sciences Center and Tsinghua University, Beijing}}

\begin{abstract}
In 1995, Ehud de Shalit proved an analogue of a conjecture of Mazur--Tate for the modular Jacobian $J_0(p)$. His main result was valid away from the Eisenstein primes. We complete the work of de Shalit by including the Eisenstein primes, and give some applications such as an elementary combinatorial identity involving discrete logarithms of difference of supersingular $j$-invariants. An important tool is our recent work on the so called ``generalized cuspidal $1$-motive''.
\end{abstract}

\section{Introduction}\label{section_introduction}

Let $p$, $\ell$ $\geq 5$ be primes such that $\ell$ divides $p-1$. Let $r=p^t$ be the largest power of $\ell$ dividing $p-1$ and $R = \mathbf{Z}/r\mathbf{Z}$. Barry Mazur and John Tate formulated in \cite{MT} an exceptional zero conjecture modulo $r$ for elliptic curves of conductor $p$. This was proved under certain technical assumptions by Ehud de Shalit in \cite{deShalit_MT} and \cite{deShalit_X_1}. de Shalit in fact proved an analogue of the Mazur--Tate conjecture for the (generalized) Jacobian of the modular curve $X_0(p)$. However, his result was not complete when localizing at the Eisenstein ideal. In this paper we resolve this issue and give various applications using the theory of the Eisenstein ideal, as we now explain in details. 

\subsection{Some notation}
We first introduce some notation. We try to follow as much as possible the notation of de Shalit, since we will frequently refer to his papers. Let $\mathscr{D}=\mathbf{Z}[\mathbf{P}^1(\mathbf{Q})]$ and $\mathscr{D}_0$ be the augmentation subgroup of $\mathscr{D}$. The group $\GL_2(\mathbf{Q})$ acts on $\mathbf{P}^1(\mathbf{Q})$ via $\begin{pmatrix}a & b \\ c & d \end{pmatrix}\cdot x = \frac{ax+b}{cx+d}$; this induces an action on $\mathscr{D}$ and $\mathscr{D}_0$. We denote by $\Gamma$, $\Gamma_0$ and $\Gamma_1$ the groups $\SL_2(\mathbf{Z})$, $\Gamma_0(p)$ and $\Gamma_1(p)$ respectively. If $G \subset \Gamma$ is a congruence subgroup and $V$ is a left $G$-module, we denote by $\Symb_{G}(V)$ the group $\Hom_G(\mathscr{D}_0, V)$ of $G$-equivariant group homomorphisms $\mathscr{D}_0 \rightarrow V$; this is called the space of modular symbols of $G$ in $V$. 

Let $\mathbb{T} = \mathbf{Z}[T_n, n\geq 1]$ be the Hecke algebra over $\mathbf{Z}$ acting on the space of modular forms of weight $2$ and level $\Gamma_0$, and by $\mathbb{T}^0$ its quotient acting on the cusp forms. There is an action of $\mathbb{T}$ on $\Symb_{\Gamma_0}(R)$ (\cf \cite[\S 2.5]{deShalit_MT} for the precise definition). 

Let $\log : (\mathbf{Z}/p\mathbf{Z})^{\times} \rightarrow R$ be a fixed surjective group homomorphism. 
The various equalities stated in this paper will be independent of the choice of $\log$ since both sides will depend on it in the same way. There is a group isomorphism $\mathbf{Q}_p^{\times} \otimes_{\mathbf{Z}} R \simeq R^2$ given by $\alpha \mapsto (v_R(\alpha), \lambda_R(\alpha))$, where $v_R(\alpha)$ is the reduction modulo $r$ of the $p$-adic valuation $\ord_p(\alpha)$ of $\alpha$ and $\lambda_R(\alpha)$ is $\log$ of the reduction modulo $p$ of $\alpha\cdot p^{-\ord_p(\alpha)} \in \mathbf{Z}_p^{\times}$.

\subsection{The Mazur--Tate conjecture for elliptic curves}

Let $E$ is an elliptic curve over $\mathbf{Q}$ of conductor $p$ with split multiplicative reduction at $p$ (\ie $a_p=1$, where $\sum_{n\geq 1 }a_n q^n \in S_2(\Gamma_0(p))$ is the eigenform associated to $E$). We know that there exists $q_E \in \mathbf{Q}_p^{\times}$, called the $p$-adic period of $E$, such that $E(\overline{\mathbf{Q}}_p) \simeq \overline{\mathbf{Q}}_p^{\times}/q_E^{\mathbf{Z}}$ as rigid analytic spaces. The Mazur--Tate conjecture may be stated as follows.

\begin{conj}[Mazur--Tate]\label{MT_conj}
For any $\psi \in \Symb_{\Gamma_0}(R)$ such that $T_q\psi = a_q\cdot \psi$ for all primes $q$, we have in $R$:
\begin{equation}\label{equation_MT_ellcurve}
\lambda_R(q_E)\cdot \psi(( 0)-(\infty)) = v_R(q_E)\cdot \sum_{a=1}^{p-1} \lambda_R(a)\cdot \psi((a/p)-(\infty)) \text{ .}
\end{equation}
\end{conj}

de Shalit proved the following result.
\begin{thm}\label{MT_deShalit_result}\cite[Theorem 0.3]{deShalit_MT}
Conjecture \ref{MT_conj} holds if the following conditions are satisfied.
\begin{enumerate}
\item \label{deShalit_assumption_torsion} $E$ has no $\ell$-torsion.
\item \label{deShalit_assumption_modular_degree}  The degree of a modular parametrization $X_0(p) \rightarrow E$ is prime to $\ell$.
\end{enumerate}
\end{thm}
\begin{rems}
\begin{enumerate}
\item As de Shalit notes in \cite[\S 6.3]{deShalit_MT}, condition (\ref{deShalit_assumption_torsion}) holds except if $E = X_0(11)$, in which case we may check conjecture \ref{MT_conj} by hand. Thus, although this condition is \textit{a posteriori} not necessary, we emphasize it because our results will allow to remove it \textit{a priori}.
\item The necessity of condition (\ref{deShalit_assumption_modular_degree}) follows from the fact that de Shalit actually works at the level of the Jacobian of $X_0(p)$. Under this condition, we have $v_R(q_E) \in R^{\times}$ so we may rewrite (\ref{equation_MT_ellcurve}) as
\begin{equation}\label{equation_MT_ellcurve_2}
\mathscr{L}_E \cdot \psi(( 0)-(\infty)) = \sum_{a=1}^{p-1} \lambda_R(a)\cdot \psi((a/p)-(\infty)) 
\end{equation}
where $\mathscr{L}_E :=v_R(q_E)^{-1}\cdot  \lambda_R(q_E) \in R$ is the so-called \textit{refined} $\mathscr{L}$-\textit{invariant of}$E$.

\item The original conjecture of Mazur and Tate was equality (\ref{equation_MT_ellcurve}) for a specific and canonical modular symbol $\psi_E$ attached to $E$ (\cf \cite[\S 6.1]{deShalit_MT}). As de Shalit notices \cite[\S 6.2]{deShalit_MT}, this is in fact not necessary and only the Hecke property $T_q\psi = a_q\cdot \psi$ is relevant.
\end{enumerate}
\end{rems}

\subsection{The Mazur--Tate conjecture for the generalized Jacobian of $X_0(p)$}

We now describe the analogue of the Mazur--Tate conjecture for (the split part of) the Jacobian $J_0$ of $X_0(p)$ (defined over $\mathbf{Q}$). The curve $X_0(p)$ has a model over $\mathbf{Z}_p$ whose special fiber is a union of two projective lines intersecting transversally at the supersingular points \cite{Deligne_Rapoport}. We denote by $S$ these supersingular points, \ie the set of isomorphism classes of supersingular elliptic curves over $\overline{\mathbf{F}}_p$. It is well-known that the elements of $S$ are defined over $\mathbf{F}_{p^2}$. Let $\mathbf{Q}_{p^2}$ be the unramified quadratic extension of $\mathbf{Q}_p$ and $\mathbf{Z}_{p^2}$ be its valuation ring. The abelian variety $J_0 \times_{\mathbf{Z}_p} \mathbf{Z}_{p^2}$ thus has split multiplicative reduction at $p$. By the general theory of Mumford curves, $J_0 \times_{\mathbf{Q}_p} \mathbf{Q}_{p^2}$ has a rigid analytic uniformization by a torus, which we now recall following \cite{deShalit}.

We let $N = \mathbf{Z}[S]$ and $N_0$ the augmentation subgroup of $N$; there is a canonical action of $\mathbb{T}$ (resp. $\mathbb{T}^0$) on $N$ (resp. $N_0$). There also is a natural action of $\Gal(\mathbf{F}_{p^2}/\mathbf{F}_p)$, and thus of $\Gal(\overline{\mathbf{Q}}_p /\mathbf{Q}_p)$, on $N$ and $N_0$.

There is a canonical bilinear pairing $Q_0 : N_0 \times N_0 \rightarrow \mathbf{Q}_{p^2}^{\times}$ inducing an isomorphism of rigid analytic spaces  over $\mathbf{Q}_{p^2}$ (hence of $\Gal(\overline{\mathbf{Q}}_p /\mathbf{Q}_{p^2})$-modules) $$J_0(p)(\overline{\mathbf{Q}}_p) \simeq \Hom(N_0,\overline{\mathbf{Q}}_p^{\times})/q_0(N_0)$$ where $q_0 : N_0 \rightarrow \Hom(N_0, \mathbf{Q}_{p^2}^{\times})$ is induced by $Q$ (note that $q_0$ is the analogue of the $p$-adic period $q_E$). In \cite{deShalit}, de Shalit extended $Q_0$ to a bilinear pairing $Q : N \times N \rightarrow \mathbf{Q}_{p^2}^{\times}$. We denote by $q : N \rightarrow \Hom(N, \mathbf{Q}_{p^2}^{\times})$ the induced linear map and by $q_R : N \otimes_{\mathbf{Z}} R \rightarrow \Hom(N \otimes_{\mathbf{Z}} R, \mathbf{Q}_{p^2}^{\times} \otimes_{\mathbf{Z}} R)$ the induced morphism modulo $r$.

We now state some of the properties of $Q$. The modular curve $X_0(p)$ has two cusps, namely the classes of $0$ and $\infty$, which are defined over $\mathbf{Q}$. We denote by $J_0^{\sharp}$ the generalized Jacobian of $X_0(p)$ with respect to the modulus $(0)+(\infty)$. We have an exact sequence of group schemes over $\mathbf{Q}$:
$$1 \rightarrow \mathbf{G}_m \rightarrow J_0^{\sharp} \rightarrow J_0 \rightarrow 1 \text{ .}$$
\begin{prop}\label{Prop_properties_Q}
\begin{enumerate}
\item \label{Prop_properties_Q_i} The pairing $v_R \circ Q : N \times N \rightarrow R$ sends $([E],[E'])$ to $\delta_{E=E'} \cdot w_E$ where $E$, $E'$ $\in S$ and $w_E \in \{1,2,3\}$ is $\frac{\Card(\Aut(E))}{2}$ ($\Aut(E)$ is the automorphism group of $E$ over $\overline{\mathbf{F}}_p$). In particular, $v_R\circ q_R : N \otimes_{\mathbf{Z}} R \rightarrow \Hom(N \otimes_{\mathbf{Z}} R, R)$ is an isomorphism.
\item \label{Prop_properties_Q_ii} The pairing $Q$ is symmetric and $\mathbb{T}$-equivariant, \ie for all $x$, $y$ in $N$ and $T \in \mathbb{T}$, we have $Q(Tx,y) = Q(Ty,x)$.
\item  \label{Prop_properties_Q_iii}  The restriction of $Q$ to $N_0 \times N$ takes values in $\mathbf{Q}_p^{\times}$ and is $\Gal(\overline{\mathbf{Q}}_p /\mathbf{Q}_p)$-equivariant, \ie for all $x \in N_0$, $y\in N$ and $\sigma \in \Gal(\overline{\mathbf{Q}}_p /\mathbf{Q}_p)$, we have $Q(\sigma x, \sigma y)=\sigma Q(x,y) = Q(x,y)$. The same is true for $Q$ (without restricting to $N_0 \times N$) modulo principal units of $\mathbf{Q}_{p^2}^{\times}$.
\item  \label{Prop_properties_Q_iv}  There is a $\mathbb{T}$ and $\Gal(\overline{\mathbf{Q}}_p /\mathbf{Q}_p)$-equivariant group isomorphism
\begin{equation}\label{uniformization_J_0}
J_0^{\sharp}(\overline{\mathbf{Q}}_p) \simeq \Hom(N, \overline{\mathbf{Q}}_p^{\times})/q(N_0) \text{ .}
\end{equation}
\end{enumerate}
\end{prop}
\begin{proof}
Assertion (\ref{Prop_properties_Q_i}) follows from \cite[\S 1.6 (9)]{deShalit}. Assertions (\ref{Prop_properties_Q_ii}) and (\ref{Prop_properties_Q_iii}) follow from \cite[Propositions 3.7 and 3.8]{Lecouturier_MMS}. Assertion (\ref{Prop_properties_Q_iv}) follows from \cite[\S 2.3]{deShalit_X_1}.
\end{proof}

\begin{defn}
The \emph{refined} $\mathscr{L}$\emph{-invariant} of weight $2$ and level $\Gamma_0(p)$ modulo $r$ is 
$$\mathscr{L}_R := (v_R\circ q_R)^{-1}\circ (\lambda_R \circ q_R) \in \End(N \otimes_{\mathbf{Z}} R) \text{ .}$$
\end{defn}

Proposition \ref{Prop_properties_Q} (\ref{Prop_properties_Q_iv}) combined with a result of Emerton \cite{Emerton_Supersingular} allows us to prove the following result.

\begin{prop}\label{L_Hecke_op_prop}
We have $\mathscr{L}_R \in \mathbb{T} \otimes_{\mathbf{Z}} R$. Thus, $\mathscr{L}_R$ acts naturally on $\Symb_{\Gamma_0}(R)$.
\end{prop}

We warn the reader that what de Shalit denotes by $\mathscr{L}_R$ in \cite{deShalit_MT} is in fact the image of our $\mathscr{L}_R$ in $\mathbb{T}^0 \otimes_{\mathbf{Z}} R$, which we denote by $\mathscr{L}_R^0$. 

Recall that in $\mathbb{T}$ we have $U_p^2=1$, where $U_p$ is the Hecke operator of index $p$. If $M$ is a $\mathbb{T} \otimes_{\mathbf{Z}} R$-module and $\epsilon \in \{1,-1\}$, we denote by $M^{U_p=\epsilon}$ the largest subspace of $M$ on which $U_p$ acts by multiplication by $\epsilon$. Note that we have $M = M^{U_p=1} \oplus M^{U_p=-1}$ since $p$ is odd. We prove the following result, which is an extension of \cite[Theorem 0.5]{deShalit}.

\begin{thm}[Main theorem]\label{main_thm}
For all $\psi \in \Symb_{\Gamma_0}(R)^{U_p=1}$, we have 
\begin{equation}\label{equation_L_invariant_main}
(\mathscr{L}_R\cdot \psi)((0)-(\infty)) = \sum_{a=1}^{p-1} \lambda_R(a)\cdot \psi((a/p)-(\infty)) \text{ .}
\end{equation}
\end{thm}

Notice the similarity between (\ref{equation_MT_ellcurve_2}) and (\ref{equation_L_invariant_main}). de Shalit in fact deduces Theorem \ref{MT_deShalit_result} from \cite[Theorem 0.5]{deShalit_MT}. Theorem \ref{main_thm} allows us to remove assumption (\ref{deShalit_assumption_torsion}) in Theorem \ref{MT_deShalit_result}. In  \cite[Theorem 0.5]{deShalit_MT}, de Shalit only proved (\ref{equation_L_invariant_main}) for modular symbols $\psi$ in a certain subspace of $\Symb_{\Gamma_0}(R)^{U_p=1}$, which is a proper subspace precisely when localizing at the Eisenstein ideal. In other words, we remove the condition ``non-Eisenstein'' in \cite[Theorem 5.6]{deShalit_MT}.

\subsection{Strategy of the proof}
We give a rough overview of the proof of Theorem \ref{main_thm}. The strategy is similar to the one of de Shalit, which is itself inspired by the proof by Ralph Greenberg and Glenn Stevens of the Mazur--Tate--Teitelbaum conjecture \cite{Greenberg_Stevens}. The key player is the Shimura covering $X_1(p) \rightarrow X_0(p)$. There are two main steps. 

The first step is to prove (\ref{equation_L_invariant_main}) where $\mathscr{L}_R$ is replaced by a certain Hecke operator which is a ``tame'' derivative (with respect to the diamond operators) of the operator $U_p^2-1$ of level $\Gamma_1(p)$. The analogous statement for $\mathscr{L}_R^0$ was tackled by de Shalit in \cite[\S 3]{deShalit_MT} using $2$-variables theta elements (``tame'' analogues of $2$-variable $p$-adic L functions).

The second step is to relate $\mathscr{L}_R$ to the above operator. We begin by briefly recalling de Shalit's strategy. By Proposition \ref{Prop_properties_Q} (\ref{Prop_properties_Q_iv}), there is a filtration of $R[\Gal(\overline{\mathbf{Q}}_p /\mathbf{Q}_p)]$-modules
\begin{equation}\label{filtration_J_0_sharp}
0 \rightarrow \Hom(N, \mu_r) \rightarrow J_0^{\sharp}[r] \rightarrow N_0 \otimes_{\mathbf{Z}} R \rightarrow 0 \text{ ,}
\end{equation}
where $\mu_r$ is the module of $r$th root of unity and $ J_0^{\sharp}[r]$ is the $r$ torsion in $J_0^{\sharp}(\overline{\mathbf{Q}}_p)$. de Shalit constructed a deformation of this filtration by considering the $r$-torsion of the generalized Jacobian $J_1^{\sharp}$ of $X_1(p)$ with respect to the reduced cuspidal modulus. This filtered deformation enabled de Shalit to relate $\mathscr{L}_R^0$ to the tame derivative of $U_p^2-1$. 

Proposition \ref{Prop_properties_Q} (\ref{Prop_properties_Q_iv}) only gives a modular interpretation of the restriction of $Q$ to $N_0 \times N$. To get information on $\mathscr{L}_R$ and not only on $\mathscr{L}_R^0$, we need to consider the full pairing $Q$. This pairing yields a $1$-motive $\mathbf{Z} = N/N_0 \rightarrow J_0^{\sharp}$. One can in fact describe algebraically this $1$-motive, and prove that it takes values in $J_0^{\sharp}(\mathbf{Q})$ \cite[Theorem 1.5]{Lecouturier_MMS}. This $1$-motive provides a Galois-equivariant extension of $N \otimes_{\mathbf{Z}} R$ by $\Hom(N, \mu_r)$, which is characterized by $\mathscr{L}_R$. The main point in the two steps above is then to construct a deformation of this extension. This is done in Theorem \ref{W_1_thm} by using the results of \cite[Theorem 1.5]{Lecouturier_MMS}, where we construct a $1$-motive $\mathbf{Z}[\cusps]_0 \rightarrow J_1^{\sharp}$. 

\subsection{Applications using Mazur's Eisenstein ideal}\label{intro_Eisenstein_ideal}

Let $I \subset \mathbb{T}$ be the Eisenstein ideal, \ie the ideal generated by the elements $T_q-q-1$ for primes $q \neq N$ and $U_N-1$. 

\begin{question}
Let $s$ be an integer such that $1 \leq s \leq t = \ord_{\ell}(p-1)$. Can we determine $$\alpha(p, \ell, s) := \sup \{n\in \mathbf{Z}_{\geq 0}, \text{ } \mathscr{L}_R \cdot (\mathbb{T} \otimes_{\mathbf{Z}} \mathbf{Z}/\ell^s\mathbf{Z}) \subset I^{n}\cdot (\mathbb{T} \otimes_{\mathbf{Z}} \mathbf{Z}/\ell^s\mathbf{Z}) \} \text{ ?}$$
\end{question}
Obviously, $\alpha(p, \ell, s)$ is a decreasing function of $s$. The first result we prove toward this question is the following.

\begin{thm}\label{thm_alpha_geq_2}
We have $\alpha(p, \ell, t) \geq 2$. 
\end{thm}
The proof is a simple combination of Theorem \ref{main_thm} and basic facts due to Mazur concerning the Eisenstein ideal. Combining Theorem \ref{thm_alpha_geq_2} and \cite[\S 1.6 Main thm]{deShalit}, one can prove the following elementary identity, for which we do not have an elementary proof.

\begin{thm}\label{corr_alpha_geq_2}
Assume $p\equiv 1 \text{ (modulo }12\text{)}$, \ie for that for all $E \in S$  we have $w_E=1$. Let $\mathcal{T}(S)$ be the set of spanning trees of the complete graph with vertices in $S$. If $T \in \mathcal{T}(S)$, let $\mathcal{E}(T)$ be the set of edges of $T$. If $E \neq E'$ are in $S$, let $[E,E']$ be the edge between $E$ and $E$. We have:
\begin{equation}\label{tree_formula}
\sum_{T \in \mathcal{T}(S)} \prod_{[E,E'] \in \mathcal{E}(T)} \log((j(E')-j(E))^{p+1})=0 \text{ .}
\end{equation}
\end{thm}

We next give a criterion for $\alpha(p, \ell, s) \geq 3$. Let $K$ be the unique extension of $\mathbf{Q}$ of degree $\ell^t$ inside $\mathbf{Q}(\zeta_p)$ and $\mathcal{O}_{K}$ be the ring of integers of $K$. Let $\mathcal{K}_s = K_2(\mathcal{O}_K[\frac{1}{\ell p}]) \otimes_{\mathbf{Z}} \mathbf{Z}/\ell^s\mathbf{Z}$, where if $A$ is a ring we denote by $K_2(A)$ the second Quillen $K$-group of $A$. There is a canonical action of $\mathbf{Z}[\Gal(K/\mathbf{Q})]$ on $\mathcal{K}_s$. If $x$, $y$ $\in \mathbf{Z}[\zeta_p, \frac{1}{\ell p}]^{\times}$, we denote by $\{x,y\} \in K_2(\mathbf{Z}[\zeta_N, \frac{1}{\ell p}])$ the corresponding Steinberg symbol and by $(x,y)$ the image of $\{x,y\}$ under the norm map $K_2(\mathbf{Z}[\zeta_p, \frac{1}{\ell p}]) \rightarrow \mathcal{K}_s$. 

\begin{thm}\label{thm_alpha_geq_3}
We have $\alpha(p, \ell, s) \geq 3$ if and only if the element
\begin{equation}\label{log_steinberg_eq}
\sum_{a \in (\mathbf{Z}/N\mathbf{Z})^{\times}} \log(a) \cdot (1-\zeta_p^a, 1-\zeta_p) \in J\cdot \mathcal{K}_s
\end{equation}
belongs to $J^2\cdot \mathcal{K}_s$, where $J$ is the augmentation ideal of $\mathbf{Z}[\Gal(K/\mathbf{Q})]$.
\end{thm}

\begin{rems}\label{rem_alpha_geq_2}
\begin{enumerate}
\item \label{remark_tame_symbol} The classical tame symbol in $K$-theory yields a group isomorphism $\mathcal{K}_s/J\cdot \mathcal{K}_s \xrightarrow{\sim} \mathbf{Z}/\ell^s\mathbf{Z}$ sending $(x,y)$ to $\log(\overline{x^{v(y)}/y^{v(x)}})$, where $v(x)$ is the $(1-\zeta_p)$-adic valuation of $x$ and the bar means the reduction modulo $(1-\zeta_p)$.
\item The proof is a consequence of Theorem \ref{main_thm} and our work (with Jun Wang) on a conjecture of Sharifi \cite{Lecouturier_Sharifi}.
\item One can show that $\sum_{a \in (\mathbf{Z}/p\mathbf{Z})^{\times}} \log(a) \cdot (1-\zeta_p^a, 1-\zeta_p)$ always belong to $J\cdot \mathcal{K}_s$. Furthermore, the group $J\cdot \mathcal{K}_s/J^2\cdot \mathcal{K}_s$ is isomorphic to $I^2\cdot (\mathbb{T} \otimes_{\mathbf{Z}} \mathbf{Z}/\ell^s\mathbf{Z})/I^3\cdot (\mathbb{T} \otimes_{\mathbf{Z}} \mathbf{Z}/\ell^s\mathbf{Z})$. The class of (\ref{log_steinberg_eq}) in $J\cdot \mathcal{K}_s/J^2\cdot \mathcal{K}_s$ corresponds to the class of $\mathscr{L}_R$ in $I^2\cdot (\mathbb{T} \otimes_{\mathbf{Z}} \mathbf{Z}/\ell^s\mathbf{Z})/I^3\cdot (\mathbb{T} \otimes_{\mathbf{Z}} \mathbf{Z}/\ell^s\mathbf{Z})$.
\item The condition $\alpha(p, \ell, s) \geq 3$ is not always satisfied, \eg if $(p,\ell,s) = (181,5,1)$. It is however satisfied if $I^2\cdot (\mathbb{T} \otimes_{\mathbf{Z}} \mathbf{Z}/\ell^s\mathbf{Z})=I^3\cdot (\mathbb{T} \otimes_{\mathbf{Z}} \mathbf{Z}/\ell^s\mathbf{Z})$, which is the case if and only if $\sum_{k=1}^{\frac{p-1}{2}} k \cdot \log(k) \not\equiv 0 \text{ (modulo }\ell \text{)}$ as a consequence of a result of Merel \cite[Th\'eor\`eme 1]{Merel_accouplement}.
\end{enumerate}
\end{rems}

\subsection*{Acknowledgements}
I would like to thank my Phd advisor Lo\"ic Merel for suggesting me to work on this problem and for his support during the completion of this paper. Part of this work began at the end of my Phd thesis and the details were worked out afterwards. This work was funded by Universit\'e Paris--Diderot, the Yau Mathematical Sciences Center, Tsinghua University and the The Fondation Sciences Math\'ematiques de Paris. 

\section{Proof of the main theorem}
In this section, we prove Theorem \ref{main_thm}. 
\subsection{Notation and conventions}
Keep the notation of section \ref{section_introduction}. 
We introduce some more notation. Fix an algebraic closure $\overline{\mathbf{Q}}$ of $\mathbf{Q}$ together with embeddings $\overline{\mathbf{Q}} \hookrightarrow \overline{\mathbf{Q}}_p$ and $\overline{\mathbf{Q}} \hookrightarrow \mathbf{C}$.

Let ${\bf \Lambda }=R[\mathbf{F}_p^{\times}]$, ${\bf J}$ be the augmentation ideal of $\bf \Lambda$, $\Lambda = R[\mu_r(\mathbf{F}_p^{\times})]$, $J$ be the augmentation ideal of $\Lambda$, $\Lambda' = R[\mu_{(p-1)/r}(\mathbf{F}_p^{\times})]$ and $J'$ be the augmentation ideal of $\Lambda'$. Here, for any integer $n\geq 1$ we have denoted by $\mu_n(\mathbf{F}_p^{\times})$ the elements of order dividing $n$ in $\mathbf{F}_p^{\times}$. Note that $\Lambda$ is a local ring and a direct factor of $\bf \Lambda$. The $R$-algebra $\Lambda'$ is \'etale.  We warn the reader that de Shalit uses the notation $\bf I$, $I$ and $I'$ for $\bf J$, $J$ and $J'$ respectively (we want to avoid confusion with the Eisenstein ideal). If $M$ is a $\Lambda$-module, we denote by $M[J]$ the elements of $M$ annihilated by $J$ (a similar notation applies to ${\bf J}$ and $J'$). We let $\langle . \rangle : \Gal(\overline{\mathbf{Q}}_p /\mathbf{Q}_p) \rightarrow \Lambda^{\times}$ be the character sending $g$ to $[\chi_p(g)^{(p-1)/r}]$ where $\chi_p : \Gal(\overline{\mathbf{Q}}_p /\mathbf{Q}_p) \rightarrow \mathbf{F}_p^{\times}$ is the cyclotomic character modulo $p$.

Concerning Hecke operators (of level $\Gamma_0$ or $\Gamma_1$), we will consider the dual ones (induced by Picard functoriality), \ie those considered in \cite[\S 1.1]{deShalit_X_1}. We warn the reader that de Shalit uses the \textit{standards} Hecke operators in \cite{deShalit_MT} (\ie those induced by Albanese functoriality). We will only use the Hecke operator $U_p$ in what follows (it is usually denoted by $U_p^*$ in the litterature, \eg in \cite[\S 5.5]{diamond_shurman}). We shall use bold letters for Hecke operators of level $\Gamma_1$, to distinguish them from those of level $\Gamma_0$ (\eg ${\bf U}_p$ versus $U_p$). We denote by $\mathbb{T}_1$ the Hecke algebra of weight $2$ and level $\Gamma_1$ over $\mathbf{Z}$. There is a ring morphism ${\bf \Lambda} \rightarrow \mathbb{T}_1 \otimes_{\mathbf{Z}} R$ sending $[a]$ to $\langle a \rangle$ (the $a$th dual diamond operator, corresponding to a matrix in $\Gamma_0$ whose \emph{upper-left} corner is congruent to $a$ modulo $p$).

We denote by $\mathcal{C}_0$ and $\mathcal{C}_1$ the cusps of $X_0$ and $X_1$ respectively. We have $\mathcal{C}_1 = \mathcal{C}_1^{\et} \sqcup \mathcal{C}_1^{\mu}$ where $\mathcal{C}_1^{\et}$ (resp. $\mathcal{C}_1^{\mu}$) is the set of cusps of $X_1$ above the cusp $0$ (resp. $\infty$) of $X_0$.

We choose the standard canonical model for the modular curve $X_1$ over $\mathbf{Q}$, \ie the moduli space of pairs $(E, \mathbf{Z}/p\mathbf{Z} \hookrightarrow E[p])$. In this model, the cusps of $X_1$, the cusps in $\mathcal{C}_1^{\mu}$ (resp. $\mathcal{C}_1^{\et}$) are defined over $\mathbf{Q}(\zeta_p)^+$ (resp. $\mathbf{Q}$).

\subsection{Construction of a filtered deformation}

Let $\overline{q} : N \rightarrow J_0^{\sharp}(\overline{\mathbf{Q}}_p)$ be the composite of $q : N \rightarrow \Hom(N, \mathbf{Q}_{p^2}^{\times})$ with the uniformization $\Hom(N, \overline{\mathbf{Q}}_p^{\times}) \rightarrow J_0^{\sharp}(\overline{\mathbf{Q}}_p)$.
We let 
$$W_0 := (J_0^{\sharp}(\overline{\mathbf{Q}}_p)/\overline{q}(N))[r] \text{ .}$$
By Proposition \ref{Prop_properties_Q}, we have a $\mathbb{T}$ and $\Gal(\overline{\mathbf{Q}}_p /\mathbf{Q}_p)$-equivariant short exact sequence 
\begin{equation}\label{fund_exact_seq_W_0}
0\rightarrow W_0^0 \rightarrow W_0 \rightarrow W_0^1 \rightarrow 0 \text{ .}
\end{equation}
where $W_0^0= \Hom(N, \mu_r)$ and  $W_0^1=N \otimes_{\mathbf{Z}} R$.

Our key input is the following result, whose proof relies on the ideas developed by the author in \cite{Lecouturier_MMS} and the techniques of de Shalit.

\begin{thm}\label{W_1_thm}
There exists a $R$-module $W_1$ with a commuting action of $\mathbb{T}_1$ and $\Gal(\overline{\mathbf{Q}}_p /\mathbf{Q}_p)$ satisfying the following properties.
\begin{enumerate}
\item \label{W_1_prop_free_lambda} $W_1$ is a free $\Lambda$-module of rank $2m+2$ where $m$ is the genus of $X_0$. 
\item \label{W_1_prop_bracket_diamond} There is a group isomorphism $\tau : W_0 \xrightarrow{\sim} W_1[J]$ which is $\mathbb{T}_1$ and $\Gal(\overline{\mathbf{Q}}_p /\mathbf{Q}_p)$-equivariant.
\item  \label{W_1_prop_filtration} There is a short exact sequence of $\mathbb{T}_1$ and $\Gal(\overline{\mathbf{Q}}_p /\mathbf{Q}_p)$ modules
\begin{equation}\label{filtration_W_1}
0 \rightarrow W_1^0 \rightarrow W_1 \rightarrow W_1^1 \rightarrow 0 \text{ ,}
\end{equation}
where $W_1^0$ and $W_1^1$ are free $\Lambda$-modules of rank $m+1$.
Furthermore, the analogous exact sequence 
\begin{equation}\label{filtration_W_1_J}
0 \rightarrow W_1^0[J] \rightarrow W_1[J] \rightarrow W_1^1[J] \rightarrow 0 \text{ .}
\end{equation}
remains exact and is identified with (\ref{fund_exact_seq_W_0}) under $\tau$.
\item\label{W_1_prop_galois_filtration}  Let $\phi : \Gal(\overline{\mathbf{Q}}_p /\mathbf{Q}_p) \rightarrow \Aut(W_1)$ be the unramified representation sending $\Frob_p$ to ${\bf U}_p$ (note that ${\bf U}_p$ acts invertibly on $W_1$ since it does on $W_1[J]$ by (\ref{W_1_prop_bracket_diamond})). The action of $\Gal(\overline{\mathbf{Q}}_p /\mathbf{Q}_p)$ on $W_1^0$ is given by $\phi^{-1}$. The action of $\Gal(\overline{\mathbf{Q}}_p /\mathbf{Q}_p)$ on $W_1^1$ is given by $\phi \langle \cdot \rangle^{-1}$. 
\end{enumerate}
\end{thm}

\begin{rems}\label{rems_main_thm}
\begin{enumerate}
\item\label{rem_1_main_thm} Theorem \ref{W_1_thm} is analogous to \cite[Theorem 4.3]{deShalit_MT}, where de Shalit deforms the filtration (\ref{filtration_J_0_sharp}) (notice that our character $\langle \cdot \rangle$ is inverse to the one of de Shalit in this theorem). The main difference here is that we have replaced (\ref{filtration_J_0_sharp}) by (\ref{fund_exact_seq_W_0}). Another important difference is that the $\Lambda$-module $W_1$ is free, while the analogous statement is false in de Shalit's situation. 
\item Theorem \ref{W_1_thm} (\ref{W_1_prop_galois_filtration}) is similar to \cite[Proposition 4.6]{deShalit_MT}.
\item There is in fact an action of $\Gal(\overline{\mathbf{Q}} /\mathbf{Q})$ on $W_0$ and $W_1$, but the exact sequences (\ref{fund_exact_seq_W_0}) and (\ref{filtration_W_1}) are not $\Gal(\overline{\mathbf{Q}} /\mathbf{Q})$-stable \textit{a priori}.
\end{enumerate}
\end{rems}
\begin{proof}
Following de Shalit, we begin by recalling some facts about the geometry of $X_1$. Let $K = \mathbf{Q}_p(\zeta_p)$ and $\mathcal{O}_K$ be its ring of integer, where $\zeta_p \in \overline{\mathbf{Q}}_p$ is a primitive $p$th root of unity. Let $\mathscr{X}_1$ be the model of $X_1$ over $\mathcal{O}_K$ considered by de Shalit in \cite[\S 2.4]{deShalit_X_1}. The special fiber over $\mathbf{F}_p$ of $\mathscr{X}_1$ is the union of two irreducible components $\Sigma^{\et}$ and $\Sigma^{\mu}$, both isomorphic to the Igusa curve $\Ig(p)$, intersecting at the supersingular points $S$. The cusps in $\mathcal{C}_1^{\et}$ (resp. $\mathcal{C}_1^{\mu}$) define $\mathcal{O}_K$ points of $\mathscr{X}_1$ whose special fibers lie in $\Sigma^{\et}$ (resp. $\Sigma^{\mu}$). 

We denote by $\mathscr{J}_1^{\sharp}$  the N\'eron model of $J_1^{\sharp} \times_{\mathbf{Q}} K$ over $\mathcal{O}_K$ and by $(\mathscr{J}_{1/\mathbf{F}_p}^{\sharp})^0$ the connected component of the special fiber $\mathscr{J}_1^{\sharp}$. We have an exact sequence of abelian group schemes over $\mathbf{F}_p$:
\begin{equation}\label{eq_def_sub}
0 \rightarrow \Hom(\mathbf{Z}[S], \mathbf{G}_{m/\mathbf{F}_p})\rightarrow (\mathscr{J}_{1/\mathbf{F}_p}^{\sharp})^0 \rightarrow J^{\et, \sharp} \times J^{\mu, \sharp}  \rightarrow 0 
\end{equation}
where $J^{\et, \sharp}$ is the generalized Jacobian of $\Sigma^{\et}$ with respect to the reduced cuspidal modulus and similarly for $J^{\mu, \sharp}$. We denote by $(\mathscr{J}_1^{\sharp})^0$ the preimage of $(\mathscr{J}_{1/\mathbf{F}_p}^{\sharp})^0(\overline{\mathbf{F}}_p)$ by the reduction map $\mathscr{J}_1^{\sharp}(\mathcal{O}_K^{\unr}) \rightarrow \mathscr{J}_1^{\sharp}(\overline{\mathbf{F}}_p)$ (where $\mathcal{O}_K^{\unr}$ is the valuation ring of the maximal unramified extension $K^{\unr}$ of $K$). Note that $(\mathscr{J}_1^{\sharp})^0$ is a subgroup of $J_1^{\sharp}(K^{\unr})$ by the N\'eron mapping property. We denote by $\mathscr{J}_1^{\sharp, \sub}$ the kernel of the projection $(\mathscr{J}_1^{\sharp})^0  \rightarrow J^{\mu, \sharp}(\overline{\mathbf{F}}_p)$. In the notation of de Shalit \cite[\S 2.7 Definition]{deShalit_X_1}, we have $\mathscr{J}_1^{\sharp, \sub}[r] = J_1^{\sharp}[r]^{\sub}$ (this follows from the facts that the kernel of the reduction map is a pro-$p$ group and $\gcd(\ell,p)=1$). 

We denote by $\mathscr{J}_1^{', \sharp}$ the N\'eron model of $J_1^{\sharp}$ over $\mathbf{Z}_p$. If $L \subset \overline{\mathbf{Q}}_p$ is an extension of $\mathbf{Q}_p$ (possibly ramified or infinite) with residue field $\kappa_L$ and valuation ring $\mathcal{O}_L$, we denote by $(\mathscr{J}_1^{', \sharp})^0(\mathcal{O}_L)$ the preimage of $(\mathscr{J}_{1/ \mathbf{F}_p}^{', \sharp})^{0}(\kappa_L)$ under the reduction map  $\mathscr{J}_1^{', \sharp}(\mathcal{O}_L) \rightarrow \mathscr{J}_{1/\mathbf{F}_p}^{', \sharp}(\kappa_L)$, where $(\mathscr{J}_{1/ \mathbf{F}_p}^{', \sharp})^{0}$ is the connected component of $\mathscr{J}_{1/\mathbf{F}_p}^{', \sharp}$. By the N\'eron mapping property, there is a canonical map of $\mathcal{O}_K$-schemes $\mathscr{J}_1^{', \sharp} \times_{\mathbf{Z}_p} \mathcal{O}_K \rightarrow \mathscr{J}_1^{\sharp}$.  We denote by $\mathcal{G}_1$ the image of $(\mathscr{J}_1^{', \sharp})^0(\mathcal{O}_K^{\unr})$ in $(\mathscr{J}_1^{\sharp})^0$ via the latter map. The argument of \cite[Proof of Proposition 2.5]{deShalit_X_1} shows that $\mathcal{G}_1$ is contained in $\mathscr{J}_1^{\sharp, \sub}$ and that $\mathcal{G}_1[r] = \mathscr{J}_1^{\sharp, \sub}[r]$.

\begin{lem}\label{Lemma_1_motive}
For $i \in \{0,1\}$, let $\mathcal{C}_i$ be the set of cusps of the modular curve $X_i$ and $\mathbf{Z}[\mathcal{C}_i]_0$ be the augmentation subgroup of $\mathbf{Z}[\mathcal{C}_i]$. There exists a group homomorphism
$$\varphi_i : \mathbf{Z}[\mathcal{C}_i]_0 \rightarrow J_i^{\sharp}(\overline{\mathbf{Q}})$$ satisfying the following properties.
\begin{enumerate}
\item\label{Lemma_1_injectivity} The map $\varphi_i$ is injective. Furthermore $\varphi_i$ is $\mathbb{T}_1$ and $\Gal(\overline{\mathbf{Q}}/\mathbf{Q})$-equivariant.

\item \label{Lemma_1_diagram}  We have a commutative diagram

\begin{tikzcd}\mathbf{Z}[\mathcal{C}_0]_0 \arrow[r, "\varphi_0"]\arrow[d]& J_0^{\sharp}(\overline{\mathbf{Q}}) \arrow[d] \\ \mathbf{Z}[\mathcal{C}_1]_0 \arrow[r, "\varphi_1"]&J_1^{\sharp}(\overline{\mathbf{Q}})\end{tikzcd}
\\
where the vertical maps are the natural traces maps.
\item \label{Lemma_1_phi_0} The following two maps $\mathbf{Z} \rightarrow J_0^{\sharp}(\overline{\mathbf{Q}}_p)$ coincide:
\begin{itemize}
\item The map $\mathbf{Z} \rightarrow J_0^{\sharp}(\overline{\mathbf{Q}}_p) $ obtained from $\varphi_0$ after identifying $\mathbf{Z}[\mathcal{C}_0]_0$ with $\mathbf{Z}$ (via the choice of $(\infty) - (0)$ as a generator) and embedding  $J_0^{\sharp}(\overline{\mathbf{Q}})$ into $J_0^{\sharp}(\overline{\mathbf{Q}}_p)$.
\item The map $\mathbf{Z} \rightarrow J_0^{\sharp}(\overline{\mathbf{Q}}_p)$ obtained from $\overline{q} : \mathbf{Z} = N/N_0 \rightarrow J_0^{\sharp}(\overline{\mathbf{Q}}_p) $.
\end{itemize}
\item \label{Lemma_1_phi_1}  The group $\varphi_1(\mathbf{Z}[\mathcal{C}_1^{\et}]_0)$ is contained in $r\cdot \mathcal{G}_1$ (and thus in $\mathscr{J}_1^{\sharp, \sub}$).
\item \label{Lemma_1_phi_1_complex_conjug} The $\mathbf{J}$-coinvariants of $(J_1^{\sharp}(\overline{\mathbf{Q}})/\varphi_1(\mathbf{Z}[\mathcal{C}_1]_0))[r]$
is a free $R$-module of rank $2m+2$ (where $m = \genus(X_0)$).
\end{enumerate}
\end{lem}
\begin{proof}
We will use the results of \cite{Lecouturier_MMS}. For the comfort of the reader, we recall some of the notation of that paper. Let $X$ be a proper smooth curve over a field $k$, $F$ be any field extension of $k$, $F(X)$ be the function field of $X \times_k F$, $\mathcal{C}$ be a subset of $X(F)$, $Y = X \backslash \mathcal{C}$ and $J^{\sharp}$ be the generalized Jacobian (over $k$) of $X$ relative to the reduced modulus with support in $\mathcal{C}$. If $P$ is a closed point in $X\times_k F$, let $F(X)_P$ be the completion of $F(X)$ at $P$ and $U_P \subset F(X)_P^{\times}$ be the group of principal unit. Let $$\Div(X, \mathcal{C})(F):=\Div(Y)(F) \oplus \bigoplus_{c \in \mathcal{C}} F(X)_{c}^{\times}/U_c^{\times}\text{ ,}$$
where $\Div(Y)(F)$ is the group of divisors of $Y$ defined over $F$. We denote by $\Div^0(X, \mathcal{C})(F)$ the kernel of the degree map $\Div(X, \mathcal{C})(F) \rightarrow \mathbf{Z}$ given by $$D \oplus (f_c \text{ modulo }U_c)_{c \in \mathcal{C}_{\Gamma}} \mapsto \deg(D) + \sum_{c \in \mathcal{C}_{\Gamma}} \ord_c(f_c) \text{ .}$$

There is a canonical map $F(X)^{\times} \rightarrow \Div^0(X, \mathcal{C})(F)$, given by
$$f \mapsto \text{div}_{Y}(f) \oplus (f \text{ modulo }U_{c})_{c \in \mathcal{C}_{\Gamma}}$$
where $\text{div}_{Y}(f)$ is the divisor of the restriction of $f$ to $Y$.
Then there is a canonical $\Gal(F/k)$-equivariant group isomorphism
\begin{equation}\label{iso_gen_jac_alg}
 J^{\#}(F)  \xrightarrow{\sim} \Div^0(X, \mathcal{C})(F)/F(X)^{\times} 
\end{equation}
sending the class of a divisor $D$ supported on $Y$ to the image of $(D \oplus 0)$ in  $\Div^0(X, \mathcal{C})(F)/F(X)^{\times}$.

We can now give the definition of $\varphi_0$ and $\varphi_1$. We start with $\varphi_0$. Note that $j^{-1}$ (resp. $(j\circ w_p)^{-1}$, where $w_p$ is the Atkin--Lehner involution) is a uniformizer at the cusp $\infty$ (resp. $0$) of $X_0$. Then $\varphi_0$ sends $(\infty)-(0)$ to the class of $j^{-1} \oplus j\circ w_p$ in $\Div^0(X_0, \mathcal{C}_0)(\mathbf{Q})/\mathbf{Q}(X_0)^{\times}$ (the latter group being identified with $J_0^{\#}(\mathbf{Q})$ by (\ref{iso_gen_jac_alg})). Property (\ref{Lemma_1_injectivity}) for $i=0$ follows from \cite[Theorem 1.5 (ii)]{Lecouturier_MMS}. Property (\ref{Lemma_1_phi_0}) follows from \cite[Theorem 1.6]{Lecouturier_MMS}.

We now define $\varphi_1$. Note that $j^{-1}$ (resp. $(j\circ w_p)^{-1}$) is a uniformizer at any cusp in $\mathcal{C}_1^{\mu}$ (resp. $\mathcal{C}_1^{\et}$). We then define $\varphi_1(\sum_{c \in \mathcal{C}_1^{\mu}} n_{\mu, c}\cdot [c] + \sum_{c \in \mathcal{C}_1^{\et}} n_{\et, c}\cdot [c])$ to be the class of $(j^{-n_{\mu,c}})_{c \in \mathcal{C}_1^{\mu}} \oplus ((j\circ w_p)^{-n_{\et,c}})_{c \in \mathcal{C}_1^{\et}}$ in $\Div^0(X_1, \mathcal{C}_1)(\mathbf{Q}(\zeta_p))/\mathbf{Q}(\zeta_p)(X_1)^{\times}$ (the latter group being identified with $J_1^{\#}(\mathbf{Q}(\zeta_p))$ by (\ref{iso_gen_jac_alg})). 

We prove property (\ref{Lemma_1_injectivity}) for $i=1$. The injectivity of $\varphi_1$ follows from \cite[Theorem 1.4 (ii), Theorem 1.1, Remarks 1 (i)]{Lecouturier_MMS}. Let us prove that $\varphi_1$ is $\Gal(\overline{\mathbf{Q}}/\mathbf{Q})$-equivariant. The action of $g \in \Gal(\overline{\mathbf{Q}}/\mathbf{Q})$ on $\Div^0(X_1, \mathcal{C}_1)(\overline{\mathbf{Q}})$ is given by the usual action on $\Div(Y_1)(\overline{\mathbf{Q}})$ and the following action on $\bigoplus_{c \in \mathcal{C}_1} F(X)_{c}^{\times}/U_c^{\times}$:
$$g\cdot (f_c)_{c \in \mathcal{C}_1} = (g\cdot f_c)_{g(c), c \in \mathcal{C}_1} \text{ .}$$
Since $j$ and $j\circ w_p$ are defined over $\mathbf{Q}$ and $\Gal(\overline{\mathbf{Q}}/\mathbf{Q})$ stabilizes $\mathcal{C}_1^{\et}$ and $\mathcal{C}_1^{\mu}$, we see that $\varphi_1$ is indeed $\Gal(\overline{\mathbf{Q}}/\mathbf{Q})$-equivariant. 

It remains to prove that $\varphi_1$ is $\mathbb{T}_1$-equivariant (recall that we have chosen the dual Hecke operators). Let $D = D^{\mu}+D^{\et} \in \mathbf{Z}[\mathcal{C}_1]_0$, where $D^{\mu} = \sum_{c \in \mathcal{C}_1^{\mu}} m_c\cdot {c} \in \mathbf{Z}[\mathcal{C}_1^{\mu}]$ and $D^{\et} = \sum_{c \in \mathcal{C}_1^{\et}} e_c\cdot {c} \in \mathbf{Z}[\mathcal{C}_1^{\et}]$. We want to prove that $\varphi_1(T\cdot D) = T\cdot \varphi_1(D)$ where $T$ is the Hecke operator $\langle a \rangle$, $\mathbf{T}_{\ell}$ or $\mathbf{U}_p$ ($a \in (\mathbf{Z}/p\mathbf{Z})^{\times}$, $\ell \neq p$ prime). 

We easily check that $\langle a \rangle \cdot \varphi_1(D)$ is the class of $$(j^{-m_{\langle a \rangle^{-1} \cdot c}}\circ \langle a \rangle^{-1})_{c \in \mathcal{C}_1^{\mu}} \oplus (j^{-e_{\langle a \rangle^{-1} \cdot c}}\circ w_p\circ \langle a \rangle^{-1})_{c \in \mathcal{C}_1^{\et}}$$ in $J_1^{\sharp}$, which is the equal to the class of 
$$(j^{-m_{\langle a \rangle^{-1} \cdot c}})_{c \in \mathcal{C}_1^{\mu}} \oplus (j^{-e_{\langle a \rangle^{-1} \cdot c}}\circ w_p)_{c \in \mathcal{C}_1^{\et}} = (j^{-m_c})_{ \langle a \rangle \cdot  c, c \in \mathcal{C}_1^{\mu}} \oplus (j^{-e_{ c}}\circ w_p)_{\langle a \rangle \cdot c, c \in \mathcal{C}_1^{\et}} = \varphi_1(\langle a \rangle \cdot D) \text{ .}$$
This proves that $\langle a \rangle \cdot \varphi_1(D) =  \varphi_1(\langle a \rangle \cdot D)$.

We now consider the Hecke operator $\mathbf{T}_{\ell}$.  If $c \in \mathcal{C}_1^{\mu}$, we have $\mathbf{T}_{\ell}(c) = (\ell\cdot\langle \ell \rangle + 1)(c)$. Similarly, if $c \in \mathcal{C}_1^{\et}$, we have $\mathbf{T}_{\ell}(c) = (\langle \ell \rangle + \ell)(c)$.  If $i \in \{0, 1, ..., \ell-1\}$, let $g_i = \begin{pmatrix} 1 & i \\ 0 & \ell \end{pmatrix}$. Let $g_{\infty} = \begin{pmatrix} \ell & 0 \\ 0 & 1 \end{pmatrix}$ and $g = \begin{pmatrix} a & b \\ c & d \end{pmatrix} \in \Gamma_1(p)$ be such that $d \equiv \ell \text{ (modulo }p\text{)}$. We easily check that $\mathbf{T}_{\ell}\cdot \varphi_1(D)$ is the class in $J_1^{\sharp}$ of $(f_c)_{c \in \mathcal{C}_1^{\mu}}\oplus (g_c)_{c \in \mathcal{C}_1^{\et}}$ where $f_c$ and $g_c$ are functions on the upper-half plane $\mathfrak{h}$ given by

$$f_c = (j^{-m_{\langle \ell \rangle^{-1} \cdot c}}\circ gg_{\infty})\cdot  \prod_{i=0}^{\ell-1}j^{-m_{c}}\circ g_i$$
and 
$$g_c = (j^{-e_{\langle \ell \rangle^{-1} \cdot c}}\circ w_p\circ  gg_{\infty})\cdot (j^{-e_{c}}\circ w_p \circ g_0)\cdot \prod_{i=1}^{\ell-1}j^{-e_{\langle \ell \rangle^{-1} \cdot c}}\circ w_p \circ g_i \text{ .}$$
Note that $j(z)^{-1} \equiv e^{2i\pi z} \text{ (modulo }U_{\infty}\text{)}$ where $z \in \mathfrak{h}$ is near $i \infty$.
Thus, we have 
$$f_{\infty}(z) \equiv (e^{2i\pi \cdot \ell z})^{m_{\langle \ell \rangle^{-1}\infty}} \cdot \prod_{i=0}^{\ell-1} (e^{2i\pi \cdot \frac{z+i}{\ell}})^{m_{\infty}}  \text{ (modulo }U_{\infty}\text{)} \text{ .}$$
We get 
$$f_{\infty} \equiv (-1)^{\ell-1} \cdot j^{-\ell\cdot m_{\langle \ell \rangle^{-1}\cdot \infty} - m_{\infty}} \text{ (modulo }U_{\infty}\text{)} \text{ .}$$
More generally, for any $c \in \mathcal{C}_1^{\mu}$ we have
$$f_{c} \equiv (-1)^{\ell-1} j^{-\ell\cdot m_{\langle \ell \rangle^{-1}\cdot c} - m_{c}} \text{ (modulo }U_{c}\text{)} \text{ .}$$
Similarly, for any $c \in \mathcal{C}_1^{\et}$ we have
$$g_c \equiv (-1)^{\ell-1} \cdot j^{-\ell\cdot e_{c} - e_{\langle \ell \rangle^{-1}\cdot c}} \circ w_p \text{ (modulo }U_{c}\text{)} \text{ .}$$
This proves that $\mathbf{T}_{\ell}\cdot \varphi_1(D) = \varphi_1(\mathbf{T}_{\ell}\cdot D)$. The proof for $\mathbf{U}_p$ is similar and left to the reader.

Property (\ref{Lemma_1_diagram}) follows from the functoriality of our construction of $\varphi_0$ and $\varphi_1$. We prove property (\ref{Lemma_1_phi_1}). By construction, the restriction of $\varphi_1$ to $\mathbf{Z}[\mathcal{C}_1^{\et}]_0$ takes values in $J_1^{\sharp}(\mathbf{Q})$ (we use the fact that $j\circ w_p : X_1(p) \rightarrow \mathbf{P}^1$ and $\mathcal{C}_1^{\et}$ are defined over $\mathbf{Q}$). By the N\'eron mapping property, $\varphi_1$ induces group homomorphisms $\varphi_1' : \mathbf{Z}[\mathcal{C}_1^{\et}]_0 \rightarrow \mathscr{J}_1^{', \sharp}(\mathbf{Z}_p) \hookrightarrow \mathscr{J}_1^{', \sharp}(\mathcal{O}_K) = (\mathscr{J}_1^{', \sharp} \times_{\mathbf{Z}_p} \mathcal{O}_K)(\mathcal{O}_K)$ and $\varphi_1'' : \mathbf{Z}[\mathcal{C}_1^{\et}]_0 \rightarrow \mathscr{J}_1^{\sharp}(\mathcal{O}_K)$ satisfying a commutative diagram:
\begin{center}
\begin{tikzcd}[column sep=small]
\mathbf{Z}[\mathcal{C}_1^{\et}]_0  \arrow{r}{\varphi_1'}  \arrow{rd}{\varphi_1''} 
  & (\mathscr{J}_1^{', \sharp} \times_{\mathbf{Z}_p} \mathcal{O}_K)(\mathcal{O}_K)\arrow{d} \\
    & \mathscr{J}_1^{\sharp}(\mathcal{O}_K)
\end{tikzcd}.
\end{center}
It remains to show that $\varphi_1'$ takes values in $r\cdot (\mathscr{J}_1^{', \sharp})^0(\mathcal{O}_K^{\unr})$. Since $ (\mathscr{J}_1^{', \sharp})^0(\mathbf{Z}_p^{\unr})$ is $r$-divisible, it suffices to prove that $\varphi_1'$ takes values in $ (\mathscr{J}_1^{', \sharp})^0(\mathbf{Z}_p)$. Let $\mathscr{J}_1'$ the N\'eron model of the (usual as opposed to generalized) Jacobian of $X_1$ over $\mathbf{Z}_p$. Define $(\mathscr{J}_1')^0(\mathbf{Z}_p) \subset \mathscr{J}_1'(\mathbf{Z}_p)$ similarly. There is a canonical map $\pi : \mathscr{J}_1^{', \sharp}(\mathbf{Z}_p) \rightarrow \mathscr{J}_1'(\mathbf{Z}_p)$ and the preimage of $(\mathscr{J}_1')^0(\mathbf{Z}_p)$ by this map is $(\mathscr{J}_1^{', \sharp})^0(\mathbf{Z}_p)$. Thus it suffices to prove that $\pi\circ \varphi_1'$ takes values in $(\mathscr{J}_1')^0(\mathbf{Z}_p)$. By construction, $\pi\circ \varphi_1'$ is the canonical map sending a divisor to its class in the Jacobian. Since all the cusps in $\mathcal{C}_1^{\et}$ lie in the same irreducible component in the special fiber, $\pi\circ \varphi_1'$ takes values in $(\mathscr{J}_1')^0(\mathbf{Z}_p)$ (the connected component of the special fiber corresponds to divisors which have degree $0$ in each irreducible component).

We finally prove property (\ref{Lemma_1_phi_1_complex_conjug}). Since $\gcd(r, 2p)=1$, \cite[Theorem 1.4 (ii) and Theorem 2.11 (vii)]{Lecouturier_MMS} shows that we have an isomorphism of $\mathbb{T}_1$-modules
$$(J_1^{\sharp}(\overline{\mathbf{Q}})/\varphi_1(\mathbf{Z}[\mathcal{C}_1]_0))[r] \simeq H_1(X_1, \mathcal{C}_1, R)_+ \oplus H_1(Y_1, R)_- \text{ .}$$
Here, $H_1(X_1, \mathcal{C}_1, R)$ is the first singular homology group of $X_1$ relative the cusps with coefficients in $R$, and similarly for $H_1(Y_1, R)$. These two groups have a natural action of the complex conjugation. The `$+$' (resp. `$-$') in the subscript means the subspace on which the complex conjugation acts by multiplication by $1$ (resp. $-1$). In order to prove property (\ref{Lemma_1_phi_1_complex_conjug}), it thus suffices to prove that the natural maps
\begin{equation}\label{map_X_1_X_0_homology}
H_1(X_1, \mathcal{C}_1, R)/\mathbf{J}\cdot H_1(X_1, \mathcal{C}_1, R) \rightarrow H_1(X_0, \mathcal{C}_0, R)
\end{equation}
and
$$H_1(Y_1, R)/\mathbf{J}\cdot H_1(Y_1, R) \rightarrow H_1(Y_0, R)$$
are isomorphisms. The former map is shown to be an isomorphism in \cite[Proposition 1.3]{deShalit_X_1} (where we use the fact that $H_1(X_1, \mathcal{C}_1, R)\simeq H^1(\Gamma_1, R)$). As for the latter map, notice it is a surjective map so it suffices to prove that $H_1(Y_1, R)/\mathbf{J}\cdot H_1(Y_1, R)$ is a free $R$-module of rank $2m+1$. There is a perfect and $\bf \Lambda$-equivariant pairing $H_1(X_1, \mathcal{C}_1, R) \times H_1(Y_1, R) \rightarrow R$ (namely the intersection pairing twisted by the Atkin--Lehner involution). In \cite[Step 1, Proof of Proposition 2.8]{deShalit_X_1}, it is shown that as a $\bf \Lambda$-module, $H_1(X_1, \mathcal{C}_1, R)$ is isomorphic to ${\bf \Lambda}^{2m}\oplus \bf \Lambda/J$. Thus, we have a $\bf \Lambda$-equivariant group isomorphism $H_1(Y_1, R) \simeq {\bf \Lambda}^{2m}\oplus \bf \Lambda/J$. This proves that $H_1(Y_1, R)/\mathbf{J}\cdot H_1(Y_1, R)$ has rank $2m+1$ over $R$. This concludes the proof of property (\ref{Lemma_1_phi_1_complex_conjug}), and thus the proof of Lemma \ref{Lemma_1_motive}.
\end{proof}

We define $$\mathbf{W}_1 = (J_1^{\sharp}(\overline{\mathbf{Q}}_p)/\varphi_1(\mathbf{Z}[\mathcal{C}_1]_0))[r] \text{, }$$
$$\mathbf{W}_1^0 = (\mathscr{J}_1^{\sharp, \sub}/\varphi_1(\mathbf{Z}[\mathcal{C}_1^{\et}]_0))[r]$$
and
$\mathbf{W}_1^1= \mathbf{W}_1/\mathbf{W}_1^0$. By Lemma \ref{Lemma_1_motive} (\ref{Lemma_1_injectivity}), $\mathbf{W}_1$ is a $\mathbb{T}_1$ and $\Gal(\overline{\mathbf{Q}}_p/\mathbf{Q}_p)$-module. Furthermore, $\mathbf{W}_1^0$ is a sub- $\mathbb{T}_1$ and $\Gal(\overline{\mathbf{Q}}_p/K)$-module. 

\begin{lem}\label{lemma_W_1^0_submodule}
The inclusion $\mathcal{G}_1 \subset \mathscr{J}_1^{\sharp, \sub}$ induces an isomorphism $$(\mathcal{G}_1/\varphi_1(\mathbf{Z}[\mathcal{C}_1^{\et}]_0))[r] \xrightarrow{\sim} W_1^0 \text{ .}$$
In particular, $\mathbf{W}_1^0$ is is stable by $\Gal(\overline{\mathbf{Q}}_p/\overline{\mathbf{Q}}_p)$. 
\end{lem}
\begin{proof}
We first note that by Lemma \ref{Lemma_1_motive} (\ref{Lemma_1_phi_1}), $\mathcal{G}_1/\varphi_1(\mathbf{Z}[\mathcal{C}_1^{\et}]_0)$ is well-defined. Let $x \in \mathscr{J}_1^{\sharp, \sub}$ such that $r\cdot x \in \varphi_1(\mathbf{Z}[\mathcal{C}_1^{\et}]_0)$. By Lemma \ref{Lemma_1_motive} (\ref{Lemma_1_phi_1}), there exists $y \in \mathcal{G}_1$ such that $r\cdot x = r\cdot y$. Thus, we have $r\cdot (x-y)=0$, so $x-y \in \mathscr{J}_1^{\sharp, \sub}[r] = \mathcal{G}_1[r]$. This proves that $x \in \mathcal{G}_1$. The second statement of the lemma follows from the fact that $\mathscr{J}_1^{', \sharp}$ is defined over $\mathbf{Z}_p$.
\end{proof}

We define $W_1 = \mathbf{W}_1[J']$, $W_1^0 = \mathbf{W}_1^0[J']$ and $W_1^1 = \mathbf{W}_1^1[J']$. Since taking the kernel by $J'$ is exact, we have an exact sequence of $\Lambda$-modules
$$0 \rightarrow W_1^0 \rightarrow W_1 \rightarrow W_1^1 \rightarrow 0 \text{ .}$$
Furthermore, this exact sequence is $\Gal(\overline{\mathbf{Q}}_p/\mathbf{Q}_p)$ and $\mathbb{T}_1$ equivariant. 

We prove Theorem \ref{W_1_thm} (\ref{W_1_prop_free_lambda}). By Lemma \ref{Lemma_1_motive} (\ref{Lemma_1_phi_1_complex_conjug}), the $R$-module $W_1/J\cdot W_1$ is free of rank $2m+2$. By Nakayama's lemma, there is a surjection of $\Lambda$-modules $\Lambda^{2m+2} \rightarrow W_1$. To prove that this surjection is an isomorphism, it suffices to prove that the rank of $W_1$ over $R$ is $r\cdot (2m+2)$. By Lemma \ref{Lemma_1_motive} (\ref{Lemma_1_injectivity}), the snake lemma provides an exact sequence of $\Lambda$-modules
\begin{equation}\label{extension_W_1}
0 \rightarrow J_1^{\sharp}(\overline{\mathbf{Q}}_p)[r, J'] \rightarrow W_1 \rightarrow (\Lambda \oplus \Lambda)_0 \rightarrow 0
\end{equation}
where $(\Lambda \oplus \Lambda)_0$ is the kernel of the composite $\Lambda \oplus \Lambda \xrightarrow{(x,y)\mapsto x+y} \Lambda \rightarrow \Lambda/J=R$ (we identify $(\Lambda \oplus \Lambda)_0$ with $R[\mathcal{C}_1]_0[{\bf J'}]$).
The $R$-module $J_1^{\sharp}(\overline{\mathbf{Q}})[r, J']$ is free of rank $2rm+1$ (this follows from the Riemann-Hurwitz formula, alternatively this follows from \cite[Theorem 1]{deShalit_X_1}). The $R$-module $(\Lambda \oplus \Lambda)_0$ is free of rank $2r-1$. This proves that the $R$-module $W_1$ is free of rank $r\cdot (2m+2)$, so this concludes the proof of Theorem \ref{W_1_thm} (\ref{W_1_prop_free_lambda}).

We prove Theorem \ref{W_1_thm} (\ref{W_1_prop_bracket_diamond}). By Lemma \ref{Lemma_1_motive} (\ref{Lemma_1_diagram}) and (\ref{Lemma_1_phi_0}), (\ref{extension_W_1}) fits into a commutative diagram of $R$-modules whose rows are exact:
\begin{center}
\begin{tikzcd}
0 \arrow[r]& J_1^{\sharp}(\overline{\mathbf{Q}}_p)[r, J']  \arrow[r] & W_1 \arrow[r]& (\Lambda \oplus \Lambda)_0 \arrow[r]&0
\\ 0 \arrow[r]& J_0^{\sharp}(\overline{\mathbf{Q}}_p)[r]  \arrow[r]\arrow[u]& W_0 \arrow[r]\arrow[u, "\tau"]& R \arrow[r]\arrow[u] &0
\end{tikzcd}.
\end{center}
By \cite[Proposition 4.2]{deShalit_MT}, the vertical map $J_0^{\sharp}(\overline{\mathbf{Q}})[r]  \rightarrow J_1^{\sharp}(\overline{\mathbf{Q}})[r, J']$ is injective and its image is $J_1^{\sharp}(\overline{\mathbf{Q}})[r, {\bf J}]$. The vertical map $R \rightarrow (\Lambda \oplus \Lambda)_0$ sends $1$ to $\sum_{\zeta \in \mu_r(\mathbf{F}_p)} [\zeta] \oplus -\sum_{\zeta' \in \mu_r(\mathbf{F}_p)} [\zeta']$, so it is injective. It follows that $\tau$ is injective. Since the image of $\tau$ is clearly contained in $W_1[J]$, it remains to show the equality. It suffices to prove that $W_1[J]$ is a free $R$-module of same rank has the one of $W_0$, namely $2m+2$. This follows from Theorem \ref{W_1_thm} (\ref{W_1_prop_free_lambda}). The map $\tau$ is $\mathbb{T}_1$ and $\Gal(\overline{\mathbf{Q}}_p/\mathbf{Q}_p)$-equivariant by construction so this proves Theorem \ref{W_1_thm} (\ref{W_1_prop_bracket_diamond}).

We prove Theorem \ref{W_1_thm} (\ref{W_1_prop_filtration}). de Shalit constructed an exact sequence (\cf \cite[\S 2.7]{deShalit_X_1}) of $\mathbb{T}_1$ and $\Gal(\overline{\mathbf{Q}}_p/\mathbf{Q}_p)$-modules:
$$0 \rightarrow J_1^{\sharp}[r]^{\sub} \rightarrow J_1^{\sharp}[r] \rightarrow J_1^{\sharp}[r]^{\quot} \rightarrow 0 \text{ .}$$ 
Furthermore, de Shalit proved (\cf \cite[Theorem 4.3]{deShalit_MT}) that we have isomorphisms of $\Lambda$-modules  $J_1^{\sharp}[r]^{\sub}  \simeq R \oplus \Lambda^m$ and $J_1^{\sharp}[r]^{\quot} \simeq \Lambda^m$.
By construction, we have $J_1^{\sharp}[r]^{\sub} = \mathscr{J}_1^{\sharp, \sub}[r]$. Thus, we have a commutative diagram whose rows and columns are exact, $\mathbb{T}_1$ and $\Gal(\overline{\mathbf{Q}}_p/\mathbf{Q}_p)$ equivariant:
\begin{center}
\begin{tikzcd}
&0 &0 &0
\\
0 \arrow[r]& R[\mathcal{C}_1^{\et}]_0[J']  \arrow[r]\arrow[u]& R[\mathcal{C}_1]_0[J'] \arrow[r]\arrow[u]& R[\mathcal{C}_1^{\mu}][J']\arrow[r] \arrow[u]&0
\\ 0 \arrow[r]& W_1^0  \arrow[r]\arrow[u] & W_1 \arrow[r]\arrow[u]& W_1^1 \arrow[r]\arrow[u]&0
\\ 0 \arrow[r]& J_1^{\sharp}[r]^{\sub}[J'] \arrow[r]\arrow[u]& J_1^{\sharp}[r][J'] \arrow[r]\arrow[u]& J_1^{\sharp}[r]^{\quot}[J'] \arrow[r]\arrow[u] &0
\\&0\arrow[u] &0\arrow[u] &0\arrow[u]
\end{tikzcd}.
\end{center}
Note that the middle column is isomorphic to (\ref{extension_W_1}). Note also that as we have isomorphisms of $\Lambda$-modules $R[\mathcal{C}_1^{\mu}][J'] \simeq \Lambda$ and $R[\mathcal{C}_1^{\et}]_0[J'] \simeq J$. Using the above diagram and the fact that $W_1$ is free over $\Lambda$, we conclude that $W_1^0$ and $W_1^1$ are free $\Lambda$-modules of rank $m+1$. We also conclude that $W_1^0[J] =J_1^{\sharp}[r]^{\sub}[J',J]=J_1^{\sharp}[r]^{\sub}[{\bf J}]$ and that $W_1^1[J]$ is an extension of $R$ by $J_1^{\sharp}[r]^{\quot}[{\bf J}]$. We know that via the isomorphism $J_0^{\sharp}[r] \xrightarrow{\sim} J_1^{\sharp}[r, {\bf J}]$, the submodule $\Hom(N, R)$ of $J_0^{\sharp}[r]$ is identified with $J_1^{\sharp}[r]^{\sub}[{\bf J}]$ (this follows again from \cite[Theorem 4.3]{deShalit_MT}). Thus, the submodule $\Hom(N, R)$ of $W_0$ is identified with $W_1^0[J]$ via $\tau$. This concludes the proof of Theorem \ref{W_1_thm} (\ref{W_1_prop_filtration}).

We finally prove Theorem \ref{W_1_thm} (\ref{W_1_prop_galois_filtration}). We first consider the Galois action on $W_1^0$. We begin by proving that $W_1^0$ is unramified. Let $x \in W_1^0$ and let $\tilde{x}$ be a lift of $x$ to $\mathscr{J}_1^{\sharp, \sub}$, so that $r\cdot \tilde{x} \in \varphi_1(\mathbf{Z}[\mathcal{C}_1^{\et}]_0)$. Let $g$ in the inertia group of $\Gal(\overline{\mathbf{Q}}_p/\mathbf{Q}_p)$. By Lemma \ref{lemma_W_1^0_submodule}, $y:=g\cdot \tilde{x}-\tilde{x}$ belongs to the the kernel of the reduction map $\rho : (\mathscr{J}_1^{\sharp})^0 \rightarrow (\mathscr{J}_{1/\mathbf{F}_p}^{\sharp})^0(\overline{\mathbf{F}}_p)$. Furthermore, we have $r\cdot y = 0$ since $r\cdot \tilde{x}$ is in $\varphi_1(\mathbf{Z}[\mathcal{C}_1^{\et}]_0) \subset J_1^{\sharp}(\mathbf{Q})$. Since $\gcd(r,p)=1$ we get $y=0$, so $W_1^0$ is unramified. The proof of \cite[Proposition \S 3.2]{deShalit_X_1} shows that $\rho(U_p\cdot \tilde{x} - p\cdot \Frob_p^{-1}\cdot \tilde{x}) = 0$ where $\Frob_p \in \Gal(\overline{\mathbf{Q}}_p/\mathbf{Q}_p)$ is any arithmetic Frobenius. Thus, we have $U_p\cdot \tilde{x} - p\cdot \Frob_p^{-1}\cdot \tilde{x} \in \Ker(\rho)$. Furthermore, we have $r\cdot (U_p\cdot \tilde{x} - p\cdot \Frob_p^{-1}\cdot \tilde{x}) = 0$ since $r\cdot \tilde{x} \in \varphi_1(\mathbf{Z}[\mathcal{C}_1^{\et}]_0)$, $U_p$ acts on $\mathcal{C}_1^{\et}$ by multiplication by $p$ and $\Frob_p$ acts trivially on $\varphi_1(\mathbf{Z}[\mathcal{C}_1^{\et}]_0)$. This proves that $U_p\cdot \tilde{x} - p\cdot \Frob_p^{-1}\cdot \tilde{x} \in \Ker(\rho)[r] = 0$, so $U_p\cdot x = p\cdot  \Frob_p^{-1}\cdot x = \Frob_p^{-1}\cdot x$ (since $p \equiv 1 \text{ (modulo }r\text{)}$). We have thus proved that the action of $\Gal(\overline{\mathbf{Q}}_p/\mathbf{Q}_p)$ on $W_1^0$ is given by $\phi^{-1}$. The fact that the action of $\Gal(\overline{\mathbf{Q}}_p/\mathbf{Q}_p)$ on $W_1^1$ is given by $\phi\langle . \rangle^{-1}$ follows from the proof of \cite[\S 3.7 Proposition]{deShalit_X_1}.

\end{proof}

\subsection{Conclusion of the proof of Theorem \ref{main_thm}}
The rest of the proof is now close to the one of de Shalit. We thus sketch the main steps and refer to his papers for further details. 

Recall that in $\mathbb{T}$, we have $U_p^2-1=0$. Since $W_1[J]=W_0$ and $W_1$ is free over $\Lambda$ by Theorem \ref{W_1_thm}, the endomorphism ${\bf U}_p^2-1$ of $W_1[J^2]$ induces an element in $\Hom_R(W_1[J^2]/W_1[J], W_1[J])$. The choice of $\log : (\mathbf{Z}/p\mathbf{Z})^{\times} \rightarrow R$ yields a generator $[d]-1$ of $J$ over $\Lambda$, where $d \in \mathbf{F}_p^{\times}$ is such that $\log(d)=1$. The multiplication by $[d]-1$ induces an identification $W_1[J^2]/W_1[J]=W_1[J]=W_0$ since $W_1$ is free over $\Lambda$. Thus ${\bf U}_p^2-1$ induces an endomorphism of $W_0$, which we denote by $U_p'$ (thought as the ``tame derivative'' of ${\bf U}_p^2-1$). By construction, $U_p'$ commutes with the action of $\mathbb{T}$ and $\Gal(\overline{\mathbf{Q}}_p/\mathbf{Q}_p)$. Furthermore, $U_p'$ stabilises $W_0^0$ and $W_0^1$ since ${\bf U}_p$ stabilises $W_1^0$ and $W_1^1$. 

The following result (and its proof) is similar to \cite[\S 4.11 Theorem]{deShalit_MT}. 

\begin{thm}\label{thm_relation_U_p_L}
We have $U_p' = \mathscr{L}_R$ in $\End_R(W_0)$.
\end{thm}
\begin{proof}
Let $M_1=W_1[J^2]$, $M_1^0 = W_1^0[J^2]$, $M_1^1 = W_1^1[J^2]$ and $\mathbb{T}_1' = (\mathbb{T}_1 \otimes_{\mathbf{Z}} R)/{\bf J}^2\cdot (\mathbb{T}_1 \otimes_{\mathbf{Z}} R)$. We have a commutative diagram of $\mathbb{T}_1'$ and $\Gal(\overline{\mathbf{Q}}_p/\mathbf{Q}_p)$-modules
\begin{equation}\label{big_diagram}
\begin{tikzcd}
\\&0\arrow[d] &0\arrow[d] &0\arrow[d]
\\
0 \arrow[r]& W_0^0  \arrow[r]\arrow[d]& W_0 \arrow[r]\arrow[d]& W_0^1 \arrow[r] \arrow[d]&0
\\ 0 \arrow[r]& M_1^0  \arrow[r]\arrow[d] & M_1 \arrow[r]\arrow[d]& M_1^1 \arrow[r]\arrow[d]&0
\\ 0 \arrow[r]& W_0^0 \arrow[r]\arrow[d]& W_0 \arrow[r]\arrow[d]& W_0^1 \arrow[r]\arrow[d] &0
\\&0 &0 &0
\end{tikzcd},
\end{equation}
where the first line (resp. the third line) corresponds to the $J$-invariants (resp. $J$-coinvariants).

\begin{lem}\label{lemma_control_Hecke}
The natural map $(\mathbb{T}_1 \otimes_{\mathbf{Z}} R)/{\bf J}\cdot (\mathbb{T}_1 \otimes_{\mathbf{Z}} R) \rightarrow \mathbb{T} \otimes_{\mathbf{Z}} R$ is an isomorphism.
\end{lem}
\begin{proof}
Note that there is indeed a natural map $\mathbb{T}_1 \rightarrow \mathbb{T}$, whose kernel contains the augmentation ideal, given by restriction to $M_2(\Gamma_0) \subset M_2(\Gamma_1)$ (where $M_2(\Gamma_i)$ is the space of modular forms over $\mathbf{Z}$ of weight $2$ and level $\Gamma_i$). Recall (\cf (\ref{map_X_1_X_0_homology})) that we have an isomorphism
\begin{equation}\label{map_X_1_X_0_homology_2}
H_1(X_1, \mathcal{C}_1, R)_+/{\bf J}\cdot H_1(X_1, \mathcal{C}_1, R)_+ \xrightarrow{\sim} H_1(X_0, \mathcal{C}_0, R)_+ \text{ ,}
\end{equation}
where the `$+$' sign means the invariant for the action of the complex conjugation (we use the fact that $\ell$ is odd). 

We claim that $H_1(X_0, \mathcal{C}_0, R)_+$ is free of rank $1$ over $\mathbb{T} \otimes_{\mathbf{Z}} R$. Since $\mathbb{T} \otimes_{\mathbf{Z}} R$ is an Artinian ring, it suffices to prove it after localizing at maximal ideals. For non maximal Eisenstein ideals, this is a consequence of \cite[Proposition 18.3]{Mazur_Eisenstein} since $\ell$ is odd. For the Eisenstein maximal ideal, this is deduced from Mazur's result in \cite[Proposition 4.1]{Lecouturier_higher}. 

We thus have a (non-canonical) commutative diagram
\begin{center}
\begin{tikzcd}
&\mathbb{T}_1 \otimes_{\mathbf{Z}} R \arrow[d]\arrow{r}{f}   & H_1(X_1, \mathcal{C}_1, R)_+ \arrow[d] 
\\ &\mathbb{T} \otimes_{\mathbf{Z}} R  \arrow[r] &H_1(X_0, \mathcal{C}_0, R)_+ 
\end{tikzcd}
\end{center}
where the bottom arrow is an isomorphism and the vertical maps are the natural ones. The map $f$ is surjective modulo ${\bf J}\cdot H_1(X_1, \mathcal{C}_1, R)_+$, and hence surjective. Since $\mathbb{T}_1 \otimes_{\mathbf{Z}} R$ and $H_1(X_1, \mathcal{C}_1, R)_+$ are finite groups of the same cardinality, $f$ is an isomorphism.  By (\ref{map_X_1_X_0_homology_2}), the kernel of $\mathbb{T}_1 \otimes_{\mathbf{Z}} R \rightarrow \mathbb{T}_0 \otimes_{\mathbf{Z}} R$ is ${\bf J}\cdot (\mathbb{T}_1 \otimes_{\mathbf{Z}} R)$.
\end{proof}

\begin{lem}\label{Lemma_freeness_Hecke_T_1}
The $\mathbb{T}_1'$-modules $M_1^0$ and $M_1^1$ are free of rank $1$.
\end{lem}
\begin{proof}
By \cite[Theorems 0.5 and 1.14]{Emerton_Supersingular}, the $\mathbb{T} \otimes_{\mathbf{Z}} R$-modules $W_0^0$ and $W_0^1$ are free of rank $1$. By Nakayama's lemma, the modules $M_1^0$ and $M_1^1$ are cyclic over $\mathbb{T}_1'$. We know by Theorem \ref{W_1_thm} (\ref{W_1_prop_filtration}) that $M_1^0$ and $M_1^1$ are free $\Lambda/J^2$-modules of rank $m+1$. By Lemma \ref{lemma_control_Hecke} and Nakayama's Lemma, $\mathbb{T}_1'$ has $m+1$ generators as a $\Lambda/J^2$. By comparison of cardinalities, $M_1^0$ and $M_1^1$ are in fact free over $\mathbb{T}_1'$.
\end{proof}

Both $U_p'$ and $\mathscr{L}_R$ stabilise $W_0^0 = \Hom(N,R)$ and $W_0^1 = N \otimes_{\mathbf{Z}} R$. 
\begin{lem}\label{Lemma_U_p'_L}
We have $U_p' = \mathscr{L}_R$ in $\End(W_0^0)$ and $\End(W_0^1)$.
\end{lem}
\begin{proof}
The proof is a standard Galois cohomological computation using (\ref{big_diagram}) and Theorem \ref{W_1_thm} (\ref{W_1_prop_galois_filtration}). We refer to the proofs of \cite[\S 4.11 and \S 5.3 Theorem]{deShalit_MT} for the structure of the proof, which is easily adaptable to our situation. In the notation of de Shalit, $U_p'$ corresponds to $-B_p$ (resp. $-C_p$) in $\End(W_0^0)$ (resp. $\End(W_0^1)$). The sign difference corresponds to the fact that de Shalit considers in \cite{deShalit_MT} diamond operators wich are dual (equivalently inverse) to our diamond operators (the latter corresponding to the ones of \cite{deShalit_X_1}), \cf Remarks \ref{rems_main_thm} (\ref{rem_1_main_thm}).
\end{proof}

We now conclude the proof of Theorem \ref{thm_relation_U_p_L}. The $\mathbb{T} \otimes_{\mathbf{Z}} R$-modules $W_0^0$ and $W_0^1$ are free of rank $1$. We can thus find a splitting (as Hecke modules) of (\ref{fund_exact_seq_W_0}) and write $W_0 \simeq W_0^0 \oplus W_0^1$. The Hecke operator $\mathscr{L}_R$ preserves this decomposition. By Lemma \ref{Lemma_U_p'_L}, it suffices to show that we can choose this splitting so that it is stable by $U_p'$. This follows from Lemma \ref{Lemma_freeness_Hecke_T_1} and the fact that ${\bf U}_p^2-1$ is in $\mathbb{T}_1'$.
\end{proof}

By definition of $W_0$ and by Lemma \ref{Lemma_1_motive}, there is a canonical Hecke and Galois equivariant embedding $J_0^{\sharp}(\overline{\mathbf{Q}}_p)[r] \hookrightarrow W_0$. Since we have fixed embeddings $\overline{\mathbf{Q}} \hookrightarrow \overline{\mathbf{Q}}_p$ and $\overline{\mathbf{Q}} \hookrightarrow \mathbf{C}$, we have canonical identifications of $\mathbb{T}$-modules
$$J_0^{\sharp}(\overline{\mathbf{Q}}_p)[r]  = J_0^{\sharp}(\overline{\mathbf{Q}})[r] = J_0^{\sharp}(\mathbf{C})[r]  = \Symb_{\Gamma_0}(R) \text{ .}$$
We refer to \cite[\S 2.4 (9)]{deShalit_MT} for the details of how these idenfications work and to \cite[\S 2.5]{deShalit_MT} for the precise definition of the action of $\mathbb{T}$ on $\Symb_{\Gamma_0}(R)$ (in this case, since Hecke operators are self-dual there are no possible confusions anyway). Note that $U_p'$ stabilises $\Symb_{\Gamma_0}(R)$ by Theorem \ref{thm_relation_U_p_L}. To conclude the proof of Theorem \ref{main_thm} it suffices to prove the following result, which is in fact an easy consequence of \cite[Proposition 3.14]{deShalit_MT}.

\begin{thm}\label{thm_relation_U_p_modSymb}
For all $\psi \in \Symb_{\Gamma_0}(R)^{U_p=1}$, we have 
\begin{equation}\label{equation_U_p'_modSymb}
(U_p' \cdot \psi)((0)-(\infty)) = \sum_{a=1}^{p-1} \lambda_R(a)\cdot \psi((a/p)-(\infty)) \text{ .}
\end{equation}
\end{thm}
\begin{proof}
Following \cite[\S 3.12]{deShalit_MT}, let $s : \Symb_{\Gamma_1}(R) \rightarrow \Symb_{\Gamma_0}(R)$ be the trace map, corresponding to the projection map $H_1(Y_1, R) \rightarrow H_1(Y_0,R)$. We apply \cite[Proposition 3.14]{deShalit_MT} with $W^0 = J_1^{\sharp}(\overline{\mathbf{Q}}_p)[r]$ in the notation of \cite[\S 3.9]{deShalit_MT}. We get that for all $\psi \in s(\Symb_{\Gamma_1}(R))^{U_p=1}$, formula (\ref{equation_U_p'_modSymb}) holds. Note that what de Shalit denotes by $-2\cdot(A_p-1)$ modulo $I^2$ is the (matrix of the) restriction of $U_p'$ to $\Symb_{\Gamma_0}(R)^{U_p=1}$ (the minus sign is explained as in the proof of Lemma \ref{Lemma_U_p'_L}). 

Note that (\ref{equation_U_p'_modSymb}) is trivially satisfied if $\psi$ is fixed by the complex conjugation (\ie for any $a, b \in \mathbf{P}^1(\mathbf{Q})$, we have $\psi((-a)-(-b)) = -\psi((a)-(b))$. Furthermore, we claim that $s$ induces a surjection $\Symb_{\Gamma_1}(R)_- \rightarrow \Symb_{\Gamma_0}(R)_-$, where the `$-$' corresponds to the elements on which the complex conjugation acts by $-1$ \ie to those $\psi$ such that $\psi((a)-(b)) = \psi((-a)-(-b))$. To prove this claim, it suffices to prove that the map $H_1(Y_1, R)_- \rightarrow H_1(Y_0,R)_-$ is surjective. This follows from the facts that $\ell$ is odd, $H_1(Y_i, R) = \Gamma_i^{\ab} \otimes_{\mathbf{Z}} R$ ($i=0,1$) and the fact that for any $\gamma \in \Gamma_0$, we have $c(\gamma)\cdot \gamma^{-1} \in \Gamma_1$ where $c(\gamma) = \begin{pmatrix} -1 & 0 \\ 0 & 1 \end{pmatrix} \gamma 
 \begin{pmatrix} -1 & 0 \\ 0 & 1 \end{pmatrix}$ corresponds to the complex conjugation. This concludes the proof of Theorem \ref{thm_relation_U_p_modSymb}.
\end{proof}

\section{Applications using Mazur's Eisenstein ideal}
Keep the notation of \S \ref{intro_Eisenstein_ideal}. Fix an integer $s$ such that $1 \leq s \leq t$. We denote by $\mathscr{L}_s$ the image of $\mathscr{L}_R$ in $\mathbb{T} \otimes_{\mathbf{Z}} \mathbf{Z}/\ell^s\mathbf{Z}$.

Let $H:=H_1(X_0(p), \cusps, \mathbf{Z}/\ell^s\mathbf{Z})$ be the singular homology of $X_0(p)$ relative to the cusps, with coefficients in $\mathbf{Z}/\ell^s\mathbf{Z}$. Similarly, we denote by $H^0$ the absolute homology $H_1(X_0(p), \mathbf{Z}/\ell^s\mathbf{Z})$. We denote by $H_+$ (resp. $(H^0)_+$) the subspace of $H$ (resp. $H^0$) fixed by the complex conjugation. If $\alpha$, $\beta \in \mathbf{P}^1(\mathbf{Q})$, we denote by $\{\alpha, \beta\} \in H$ the class of the geodesic path between $\alpha$ and $\beta$ in the Poincar\'e upper-half plane. 

\subsection{Results concerning the Eisenstein ideal}\label{section_reminder_Eisenstein}

In this paragraph, we recall some results concerning the Eisenstein ideal that we will use to prove the theorems stated in \S \ref{intro_Eisenstein_ideal}. 

By Manin, we have a surjective group homomorphism $\xi : \mathbf{Z}[\Gamma_0\backslash \SL_2(\mathbf{Z})] \rightarrow H$
sending $\Gamma_0\cdot \begin{pmatrix} a & b \\ c & d \end{pmatrix}$ to $\{\frac{b}{d}, \frac{a}{c}\}$. We identify $\Gamma_0\backslash \SL_2(\mathbf{Z})$ with $\mathbf{P}^1(\mathbf{Z}/p\mathbf{Z})$, via the map $\Gamma_0\cdot \begin{pmatrix} a & b \\ c & d \end{pmatrix} \mapsto [c:d]$. The subgroup $H^0$ of $H$ is spanned by the elements $\xi([a:1])$ with $a \not\equiv 0 \text{ (modulo }p\text{)}$. Since $[a:1]$ corresponds to the coset $\Gamma_0\cdot \begin{pmatrix} 1 & 0 \\ a & 1 \end{pmatrix}$, the group $H^0$ is spanned by the elements $\{0, \frac{1}{a}\}$ for $a \in \mathbf{Z}$ coprime to $p$ (this only depends on $a$ modulo $p$). Let $w_p \in \mathbb{T}$ be the Atkin--Lehner involution (which, in this case, is simply $-U_p$). We have $w_p\{0, \frac{1}{a}\} = \{\infty, -\frac{a}{p}\}$, so $H^0$ is also spanned by the elements $\{\infty, \frac{a}{p}\}$ for $a \in \mathbf{Z}$ coprime to $p$. In view of (\ref{equation_L_invariant_main}), it will be more convenient to work with the latter elements rather than the original Manin symbols.

\begin{thm}\label{thm_Eisenstein_ideal}
\begin{enumerate}
\item\label{thm_Eisenstein_ideal_i} We have $I \cdot H_+ = (H^0)_+$.
\item \label{thm_Eisenstein_ideal_ii} (Mazur) There is a group isomorphism $(H^0)_+/I\cdot (H^0)_+ \xrightarrow{\sim} \mathbf{Z}/\ell^s\mathbf{Z}$ sending $\sum_{a \in (\mathbf{Z}/p\mathbf{Z})^{\times}} \lambda_{a}\cdot \{\infty, \frac{a}{p}\}$ to $\sum_{a \in (\mathbf{Z}/p\mathbf{Z})^{\times}} \lambda_{a}\cdot \log(a)$ modulo $\ell^s$.
\item (Wang--Lecouturier) \label{thm_Eisenstein_ideal_iii} There is a group isomorphism $I\cdot (H^0)_+/I^2\cdot (H^0)_+ \xrightarrow{\sim} J\cdot \mathcal{K}_s/J^2\cdot \mathcal{K}_s$ sending $\sum_{a \in (\mathbf{Z}/p\mathbf{Z})^{\times}} \lambda_{a}\cdot \{\infty, \frac{a}{p}\}$ to 
$$\sum_{a \in (\mathbf{Z}/p\mathbf{Z})^{\times}} \lambda_{a}\cdot \left((1-\zeta_p^a, 1-\zeta_p)-\frac{1}{2}\cdot ([\sigma_a]-1)\cdot (1-\zeta_p^a, 1-\zeta_p) \right) \text{ ,}$$
where $\sigma_a \in \Gal(K/\mathbf{Q})$ is the restriction of the automorphism of $\mathbf{Q}(\zeta_p)$ sending $\zeta_p$ to $\zeta_p^a$.
\end{enumerate}
\end{thm}
\begin{proof}
Point (\ref{thm_Eisenstein_ideal_i}) is well-known, but we give a proof for the convenience of the reader. We have an exact sequence of $\mathbb{T}$-modules
$$0 \rightarrow H^0 \rightarrow H \rightarrow (\mathbf{Z}/\ell^s\mathbf{Z})[\mathcal{C}_0]_0 \rightarrow 0$$
where $(\mathbf{Z}/\ell^s\mathbf{Z})[\mathcal{C}_0]_0$ is the group of degree zero divisors supported on the cusps $\mathcal{C}_0 = \{\Gamma_0 0, \Gamma_0 \infty\}$. Since $I$ annihilates $(\mathbf{Z}/\ell^s\mathbf{Z})[\mathcal{C}_0]_0 $, we have $I\cdot H_+ \subset (H^0)_+$. To prove the equality, it suffices to prove that $H_+/I\cdot H_+ \simeq \mathbf{Z}/\ell^s\mathbf{Z}$. This is a consequence of a result of Mazur, as explained in \cite[Proposition 4.1]{Lecouturier_higher}. Point (\ref{thm_Eisenstein_ideal_ii}) follows from \cite[Proposition II.18.8]{Mazur_Eisenstein}. Point (\ref{thm_Eisenstein_ideal_iii}) follows from \cite[Theorem 1.7]{Lecouturier_Sharifi}.
\end{proof}

The following result, due to Mazur, will be useful.
\begin{prop}\label{Prop_criterion_Eisenstein}
Let $U \in \mathbb{T} \otimes_{\mathbf{Z}} \mathbf{Z}/\ell^s\mathbf{Z}$ and $n \in \mathbf{Z}_{\geq 0}$. The following assertions are equivalent.
\begin{enumerate}
\item $U \in I^n \cdot (\mathbb{T} \otimes_{\mathbf{Z}} \mathbf{Z}/\ell^s\mathbf{Z})$.
\item $U\cdot H_+ \subset I^n \cdot H_+$.
\end{enumerate}
\end{prop}
\begin{proof}
This follows from the fact that $H_+$ is locally free of rank $1$ at the maximal ideal $I+(\ell)$, a consequence of a result of Mazur (\cf \cite[Proposition 4.1]{Lecouturier_higher}).
\end{proof}

\subsection{Study of $\alpha(p, \ell, s)$}

We first reformulate Theorem \ref{main_thm} using the language of \S \ref{section_reminder_Eisenstein}. There is a perfect bilinear pairing $\bullet : \Symb_{\Gamma_0}(\mathbf{Z}/\ell^s\mathbf{Z})\times H \rightarrow \mathbf{Z}/\ell^s\mathbf{Z}$ given by $\psi \bullet \{\alpha, \beta\} = \psi((\alpha)-(\beta))$. This pairing is $\mathbb{T}$-equivariant, meaning that for all $T \in \mathbb{T}$, $\psi \in  \Symb_{\Gamma_0}(\mathbf{Z}/\ell^s\mathbf{Z})$ and $x \in H$ we have $(T\cdot\psi) \bullet x = \psi\bullet (T\cdot x)$ (in general the pairing exchanges $T$ with its dual $T^*$, but in level $\Gamma_0$ we have $T=T^*$).

Theorem \ref{main_thm} is equivalent to the single equality in $H_+$:
\begin{equation}\label{reformulation_main_thm}
\mathscr{L}_s\cdot \{0, \infty\} = \frac{1}{2}\cdot (U_p+1)\cdot \sum_{a \in (\mathbf{Z}/p\mathbf{Z})^{\times}} \log(a)\cdot \{\frac{a}{p}, \infty\} 
\end{equation}
The factor $\frac{1}{2}\cdot (U_p+1)$ is to take into account the fact that Theorem \ref{main_thm} is restricted to those $\psi$ fixed by $U_p$ (recall that $U_p=-w_p$ so $U_p\{0, \infty\} = -w_p\{0, \infty\} = -\{\infty,0\} = \{0, \infty\}$). 

We now prove Theorem \ref{thm_alpha_geq_2}. Notice that the right hand side of (\ref{reformulation_main_thm}) lies in $(H^0)_+$. By Proposition \ref{Prop_criterion_Eisenstein} and Theorem \ref{thm_Eisenstein_ideal} (\ref{thm_Eisenstein_ideal_i}), we have $\mathscr{L}_s \in I\cdot (\mathbb{T} \otimes_{\mathbf{Z}} \mathbf{Z}/\ell^s\mathbf{Z})$. To prove that we have $\mathscr{L}_s\in I^2\cdot (\mathbb{T} \otimes_{\mathbf{Z}} \mathbf{Z}/\ell^s\mathbf{Z})$, it suffices to prove that we have $\sum_{a \in (\mathbf{Z}/p\mathbf{Z})^{\times}} \log(a)\cdot \{\frac{a}{p}, \infty\} \in I\cdot (H^0)_+$. This follows from Theorem \ref{thm_Eisenstein_ideal} (\ref{thm_Eisenstein_ideal_ii}) and the fact that $\sum_{a \in (\mathbf{Z}/p\mathbf{Z})^{\times}} \log(a)^2\equiv 0 \text{ (modulo }\ell^s\text{)}$.

We now prove Theorem \ref{thm_alpha_geq_3}. By (\ref{reformulation_main_thm}) and Proposition \ref{Prop_criterion_Eisenstein}, we have $\alpha(p, \ell, s)\geq 3$ if and only if $$\frac{1}{2}\cdot (U_p+1)\cdot \sum_{a \in (\mathbf{Z}/p\mathbf{Z})^{\times}} \log(a)\cdot \{\frac{a}{p}, \infty\} \in  I\cdot (H^0)_+ \text{ .}$$
Since $\ell$ is odd and $U_p-1 \in \cap_{n\geq 0} I ^n$, we have $\alpha(p, \ell, s)\geq 3$ if and only if $$ \sum_{a \in (\mathbf{Z}/p\mathbf{Z})^{\times}} \log(a)\cdot \{\frac{a}{p}, \infty\} \in  I\cdot (H^0)_+ \text{ .}$$
By Theorem \ref{thm_Eisenstein_ideal} (\ref{thm_Eisenstein_ideal_iii}), this is equivalent to saying that 
$$\sum_{a \in (\mathbf{Z}/p\mathbf{Z})^{\times}} \log(a) \cdot (1-\zeta_p^a, 1-\zeta_p) -\frac{1}{2}\cdot ([\sigma_a]-1)\cdot (1-\zeta_p^a, 1-\zeta_p) \in J^2\cdot \mathcal{K}_s \text{ .}$$
To conclude the proof of Theorem \ref{thm_alpha_geq_3}, it suffices to prove that 
$$\sum_{a \in (\mathbf{Z}/p\mathbf{Z})^{\times}} \log(a)\cdot ([\sigma_a]-1)\cdot (1-\zeta_p^a, 1-\zeta_p) \in J^2\cdot \mathcal{K}_s \text{ ,}$$
or equivalently that
$$\sum_{a \in (\mathbf{Z}/p\mathbf{Z})^{\times}} \log(a)^2 \cdot (1-\zeta_p^a, 1-\zeta_p) \in J\cdot \mathcal{K}_s \text{ .}$$
By Remarks \ref{rem_alpha_geq_2} (\ref{remark_tame_symbol}), this is equivalent to $\sum_{a \in (\mathbf{Z}/p\mathbf{Z})^{\times}} \log(a)^3 \equiv 0 \text{ (modulo }\ell^s\text{)}$, which is true since $\ell \geq 5$.

\subsection{The tree formula}
This this last paragraph, we prove Theorem \ref{corr_alpha_geq_2}. We recall the main result of \cite{deShalit}, that are essential to our proof. Recall that by definition, $\mathscr{L}_R$ is an endomorphism of $N\otimes_{\mathbf{Z}} R$. Joseph Oesterl\'e conjectured an explicit formula for this endomorphism, and de Shalit proved this conjecture in \cite{deShalit} (up to a sign, which is unimportant here since $\ell$ is odd). The formula is as follows (\cf  \cite[\S 1.6 Main thm]{deShalit}).

For any $E \in S$, we have in $N \otimes_{\mathbf{Z}}R$:
\begin{equation}
\mathscr{L}_R([E]) = \sum_{E' \in S \atop E' \neq E} \frac{1}{w_{E'}} \cdot \log((j(E')-j(E))^{p+1})\cdot ([E']-[E]) \text{ ,}
\end{equation}
where $j(E)$ is the $j$-invariant of $E$. Recall that $j(E)$ is in $\mathbf{F}_{p^2}$, so $(j(E')-j(E))^{p+1}$ belongs to $\mathbf{F}_p^{\times}$ and we can apply $\log$ to it. The image of $\mathscr{L}_R$ is visibly contained in $N_0 \otimes_{\mathbf{Z}} R$, which is another way to prove that $\mathscr{L}_R \in I\cdot (\mathbb{T} \otimes_{\mathbf{Z}} R)$. 

Using results of Mazur and Emerton, one can show that the fact that $\alpha(p, \ell, t) \geq 2$ implies that the kernel of $\mathscr{L}_R$ (in $N \otimes_{\mathbf{Z}} R$) contains a free $R$-module of rank $2$. More precisely, $\Ker(\mathscr{L}_R)$ contains $R\cdot e_0 \oplus R\cdot e_1$ where $e_0$ and $e_1$ are the first two \textit{higher Eisenstein elements} in $N \otimes_{\mathbf{Z}} R$ defined in \cite[\S 3.1]{Lecouturier_higher}. We have an explicit elementary formula for $e_1$ \cite[Theorem 1.6]{Lecouturier_higher}, which gives $m+1$ quadratic identities in $\log$ of difference of supersingular $j$-invariants, where $m = \genus(X_0)$. We leave to the interested reader the task to write down the formulas, since we focus instead on the single formula (\ref{tree_formula}).

Fix an ordering $(E_0, ..., E_m)$ of $S$ and denote by $M$ the matrix of $\mathscr{L}_R$ in the $R$-basis $([E_0], ..., [E_m])$ of $N \otimes_{\mathbf{Z}} R$. Note that $M$ is symmetric if $p\equiv 1 \text{ (modulo }12\text{)}$, which we assume from now on. We have seen that $\Ker(M)$ contains a sub-$R$-module isomorphic to $R^2$. We can thus apply the following fact (whose proof is easy and left to the reader).

\begin{lem}
Let $q$ be a prime power and $M$ be a square matrix with coefficients in $\mathbf{Z}/q\mathbf{Z}$ whose kernel contains a submodule isomorphic to $(\mathbf{Z}/q\mathbf{Z})^2$. Then all the first minors of $M$ are equal to $0$.
\end{lem}

To conclude the proof of Theorem \ref{corr_alpha_geq_2}, it suffices to prove that the left-hand side of (\ref{tree_formula}) is (up to sign) the determinant of a first minor of $M$. This is a direct consequence of the weighted matrix-tree theorem (a generalization of Kirchhoff's theorem) applied to the complete graph whose vertices are elements of $S$ and whose weight on the edge linking $E$ and $E'$ is $\log((j(E')-j(E))^{p+1})$. Note that we have used the fact that $M$ is symmetric and that the sum of each column is zero, which is why we assumed $p \equiv 1 \text{ (modulo }12\text{)}$.

\bibliography{biblio}
\bibliographystyle{plain}
\newpage

\end{document}